\documentclass[letterpaper,10pt,oneside,reqno]{article}

\usepackage[sorting=nyt,style=alphabetic,backend=bibtex,hyperref=true,doi=false,maxbibnames=9,maxcitenames=4,eprint=false]{biblatex}
\makeatletter
\def\blx@maxline{77}
\makeatother
\usepackage{amsmath,amssymb,amsthm,amsfonts}
\usepackage{graphicx,color}
\usepackage{upgreek}
\usepackage[mathscr]{euscript}

\allowdisplaybreaks
\numberwithin{equation}{section}

\usepackage{tikz}
\usetikzlibrary{shapes,arrows,positioning,decorations.markings}

\usepackage{array}
\usepackage{adjustbox}
\usepackage[colorlinks=true,linkcolor=blue,citecolor=red]{hyperref}
\usepackage{cleveref}
\usepackage{enumerate}

\usepackage[DIV=12]{typearea}

\newcommand{\hnu}{\boldsymbol\upnu}
\newcommand{\hmu}{\boldsymbol\upmu}
\newcommand{\eps}{\epsilon}

\newtheorem{proposition}{Proposition}[section]
\newtheorem{lemma}[proposition]{Lemma}

\newtheorem{theorem}[proposition]{Theorem}
\newtheorem{conjecture}[proposition]{Conjecture}
\theoremstyle{definition}
\newtheorem{definition}[proposition]{Definition}
\newtheorem{remark}[proposition]{Remark}

\begin{document}
\title{The $q$-Hahn PushTASEP}
\author{Ivan Corwin, Konstantin Matveev, and Leonid Petrov}

\date{}

\maketitle

\begin{abstract}
	We introduce the $q$-Hahn PushTASEP --- an integrable stochastic interacting particle system which is a 3-parameter generalization of the PushTASEP, a well-known close relative of the TASEP (Totally Asymmetric Simple Exclusion Process). The transition probabilities in the $q$-Hahn PushTASEP are expressed through the $_4\phi_3$ basic hypergeometric function. Under suitable limits, the $q$-Hahn PushTASEP degenerates to all known integrable (1+1)-dimensional stochastic systems with a pushing mechanism. One can thus view our new system as a pushing counterpart of the $q$-Hahn TASEP introduced by Povolotsky \cite{Povolotsky2013}. We establish Markov duality relations and contour integral formulas for the $q$-Hahn PushTASEP. 

	In a $q\to 1$ limit of our process we arrive at a random recursion which, in a special case, appears to be similar to the inverse-Beta polymer model. However, unlike in recursions for Beta polymer models, the weights (i.e., the coefficients of the recursion) in our model depend on the previous values of the partition function in a nontrivial manner.
\end{abstract}

\section{Introduction}
\label{sec:intro}

Integrable probability is a field which seeks to discover and analyze probabilistic systems
enjoying significant algebraic structure (e.g. Markov dualities and Bethe ansatz diagonalizability) and which demonstrate universal asymptotic
fluctuation behaviors.
There are two main routes which have proved effective so far
in producing integrable probabilistic systems
--- Macdonald processes (and various degenerations)
\cite{BorodinCorwin2011Macdonald}, \cite{BorodinGorinSPB12}, \cite{BorodinPetrov2013Lect}, and
stochastic vertex models
\cite{Borodin2014vertex}, \cite{CorwinPetrov2015}, \cite{BorodinPetrov2016_Hom_Lectures}.

Prototypical examples within both classes of integrable models
are the TASEP (Totally Asymmeric Simple Exclusion Process)
\cite{Spitzer1970}
and the PushTASEP, its counterpart with a long-range pushing mechanism
\cite{liggett1980long}, \cite{derrida1991dynamics}.
In the context of \emph{Macdonald processes}, both systems are connected to certain
probability distributions on integer partitions
$\lambda=(\lambda_1\ge\ldots\ge\lambda_N\ge0 )$
with nice algebraic structure,
and the TASEP and PushTASEP are related to the distributions of,
respectively, the smallest
and the largest
parts of $\lambda$.
This picture can be generalized within the Macdonald process hierarchy to produce 2-parameter $q$-deformed discrete and continuous time TASEPs and PushTASEPs
\cite{BorodinCorwin2013discrete},
\cite{BorodinPetrov2013NN},
\cite{CorwinPetrov2013},
\cite{MatveevPetrov2014}.
On the other hand,
TASEP also belongs to the class
of exactly solvable
\emph{stochastic vertex models},
and hence it admits a 3-parameter generalization
to the $q$-Hahn TASEP
\cite{Povolotsky2013},
\cite{Corwin2014qmunu}
(recalled in \Cref{sec:TASEP}).
The latter is, in fact, a special case of the 4-parameter family of
stochastic higher spin six vertex models studied in
\cite{CorwinPetrov2015}, \cite{BorodinPetrov2016_Hom_Lectures}.

For some time it was not clear how to extend
the PushTASEP analogously outside the Macdonald hierarchy. In this paper we achieve this.
That is, we introduce a 3-parameter $q$-Hahn generalization of the
PushTASEP (\Cref{sub:def_nonnegativity}), adding a new parameter
to the $q$-PushTASEP with $q$-geometric jumps
discovered in \cite{MatveevPetrov2014}.
This is akin to the $q$-Hahn extension
of the discrete time
$q$-TASEP \cite{BorodinCorwin2013discrete}
found in
\cite{Povolotsky2013} and studied in \cite{Corwin2014qmunu}.
Our pushing system enjoys a variant of Markov duality
(see \Cref{thm:push_duality})
similar to the one studied earlier for the continuous time $q$-PushTASEP
\cite{CorwinPetrov2013},
and
generalizes the
contour integral formulas
from that setting, too (see \Cref{thm:push_q_moments_in_the_text}).
The system itself is not so simple (in particular, its transition probabilities
are expressed through the basic hypergeometric function $_4\phi_3$), and
our search for it was informed by the desire to extend the duality and
contour integral formulas away from the Macdonald context.
Even once this generalized system is introduced (see \Cref{sub:def_nonnegativity}), it
is not so simple to prove the duality, as it requires some interesting
$q$-identities (the central statement is \Cref{lemma:main_identity} which we prove in
\Cref{sub:lemma_proof}).

The $q$-Hahn TASEP revealed some interesting new systems as its limits --- in
particular, the beta polymer (or random walk in random environment) considered in
\cite{CorwinBarraquand2015Beta} which arises in a scaling limit of the system as $q\to1$.
We parallel this limit here, and discover a corresponding beta polymer-like system.
The main new feature is that in our
model, the distribution of the weights depends on the ratio of the
partition functions immediately to the left and below
(see \Cref{def:Z_process}).
We had initially expected that the inverse beta polymer of
\cite{thieryLD2015integrable} would arise from our pushing system, but
presently this does not seem to be the case
(though perhaps the inverse beta polymer could be included
into a 4-parameter family of pushing systems whose
existence we speculate on below). However, in the $q\to 1$ limit we do arrive (see \Cref{lem:Ztilde} and \Cref{thm:beta_convergence}) at a solution to the following recursion relation, which bears great similarity to that satisfied by the inverse beta polymer: When $\tilde{Z}(i,t-1)>\tilde{Z}(i-1,t)$,
$$
\tilde{Z}(i,t) = \tilde{Y} \tilde{Z}(i,t-1) + (1-\tilde{Y}) \tilde{Z}(i-1,t)
$$
where $\tilde{Y}$ is $\mathscr{Beta}^{-1}(\bar{\mu},\tfrac{1}{2}-\bar{\mu})$-distributed (see \Cref{sub:2F1_distributions}); and when
$\tilde{Z}(i,t-1)<\tilde{Z}(i-1,t)$,
$$
\tilde{Z}(i,t) = \tilde{Y} \tilde{Z}(i-1,t) + (1-\tilde{Y}) \tilde{Z}(i,t-1)
$$
where $\tilde{Y}$ is $\mathscr{Beta}^{-1}(\bar{\mu},\tfrac{1}{2})$-distributed. For this random recursion (with suitable boundary data) we compute moment formulas (\Cref{prop:beta_moments}) and conjecture a Laplace transform formula (\Cref{conj:conjecture_push_Fredholm_beta}).

Besides the $q\to 1$ limits, there are other new limit processes which arise from our work. In particular, for $q=0$ we find a geometric-Bernoulli generalization of the PushTASEP (see \Cref{sub:gB_push}).

We do not perform any asymptotics in this paper. In fact, the key result towards such an aim would be a Fredholm determinant formula for a suitable $q$-Laplace transform. However, due to the pushing mechanism, the moments generally used to derive such a formula become infinite past a certain power. Still, we present \Cref{conj:conjecture_push_Fredholm} which contains what we believe is a correct formula based on previous analogous works.
\medskip

Our present investigation suggests that there should be strong parallels
between pushing and non-pushing integrable particle systems.
In the non-pushing (e.g. standard TASEP) context,
\cite{CorwinPetrov2015}
introduced a 4-parameter family of stochastic vertex
models,
which recovers the 3-parameter $q$-Hahn TASEP
in a special analytic continuation and degeneration.
It would be
interesting to develop a parallel 4-parameter family of pushing systems,
and obtain the corresponding duality relations and contour integral formulas.
We expect that this can be done in the following manner. Consider the stochastic higher spin vertex model with horizontal spin $J\in \mathbb{Z}_{\ge1}$.
Subtract $J$ from the number of  arrows along  each horizontal edge,
and interpret the
negative arrow numbers as counting arrows pointing in the opposite direction.
This minor modification introduces a pushing mechanism.
It seems likely that duality and contour integral formulas
can be carried over from the original stochastic higher spin model.
On the other hand, to recover the $q$-Hahn PushTASEP one must find the
right analytic continuation.
We leave this for a later investigation.

We also note that it should be possible to produce a 4-parameter family of pushing
systems by a different mechanism --- via bijectivisation of the Yang-Baxter equation (see \cite{BufetovPetrovYB2017})
related to the spin $q$-Whittaker polynomials (introduced in \cite{BorodinWheelerSpinq}).
This approach is developed in the upcoming work
\cite{BufetovMucciconiPetrov2018}, and we anticipate that the same 4-parameter family will come up in this manner.

\subsection*{Outline}

In \Cref{sec:TASEP} we recall
duality and contour integral formulas for the $q$-moments of
the $q$-Hahn TASEP
\cite{Povolotsky2013}, \cite{Corwin2014qmunu}, as some of these ingredients are
used for the $q$-Hahn PushTASEP.
In \Cref{sec:q_hahn_push} we introduce the $q$-Hahn PushTASEP,
discuss its various degenerations, and prove duality and contour integral formulas.
In \Cref{sec:beta_limit} we consider a beta polymer-like limit
of the $q$-Hahn PushTASEP as $q\to1$, and write down moments of the resulting system
in a contour integral form. We also provide a conjecture for the $q$-Laplace transform.
Formulas pertaining to $q$-hypergeometric functions
and associated probability distributions on $\mathbb{Z}$
are summarized in
\Cref{sec:q_hyp}.

\subsection*{Acknowledgments}

The authors thank Guillaume Barraquand for pointing out the simplification which led to \Cref{sec:cov}. 
The authors acknowledge that part of this work was done while in attendance of the ICERM conference on Limit Shapes, held from April 13-17, 2015.
IC was partially supported by the NSF grants DMS-1208998, DMS-1811143 and DMS-1664650, as well as a Packard Fellowship in Science and Engineering and a Clay Research Fellowship.
LP was partially supported by the NSF grant DMS-1664617.

\section{$q$-Hahn TASEP}
\label{sec:TASEP}

Here we briefly recall the definition and duality properties
of the $q$-Hahn TASEP introduced and studied in
\cite{Povolotsky2013}, \cite{Corwin2014qmunu}.

Assume that the parameters $0<q<1$, $0\le \hnu\le \hmu<1$
are fixed.\footnote{Note that we are using a different font for the
$q$-Hahn TASEP parameters $(\hmu,\hnu)$
to distinguish them from the parameters $(\mu,\nu)$ of the
$q$-Hahn PushTASEP.}
The $q$-Hahn TASEP is a discrete time (with $t\in \mathbb{Z}_{\ge0}$) Markov
process on configurations
$\vec x(t)=(x_1(t)>x_2(t)>\ldots )$, $x_i\in\mathbb{Z}$, with at most one particle per site and a rightmost particle $x_1$.
The evolution of the $q$-Hahn TASEP is as follows.
At each discrete time step $t\to t+1$, each particle $x_i(t)$
jumps in parallel and independently to
$x_i(t+1)=x_{i}(t)+v_i$, where
$v_i$ is sampled from the probability distribution
$\varphi_{q,\hmu,\hnu}(v_i\mid x_{i-1}(t)-x_i(t)-1)$. Note for the update of $x_1$, we assume a virtual particle $x_0\equiv +\infty$, by agreement.
Here
$\varphi_{q,\hmu,\hnu}$
is the $q$-beta-binomial distribution
defined by \eqref{phi_qmunu_definition}.
See
\Cref{fig:TASEP} for an illustration.
\begin{figure}[htpb]
	\centering
	\includegraphics{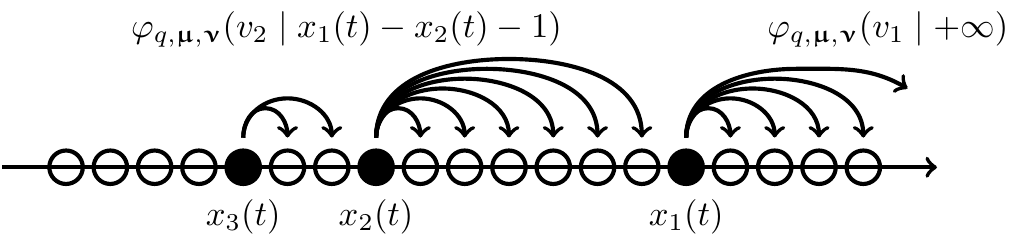}
	\caption{$q$-Hahn TASEP and the update probabilities
	$\varphi_{q,\hmu,\hnu}$ for $x_1$ and $x_2$.}
	\label{fig:TASEP}
\end{figure}

We now proceed to describe a Markov duality relation between
the $q$-Hahn TASEP one-step transition operator
and  the $q$-Hahn Boson operator on a different space.
Fix $N\ge1$ and define the $N$-particle space
\begin{equation}\label{space_X_N}
	\mathbb{X}^N:=
	\{\vec x:=(x_0,x_1,\ldots,x_N ), \ x_i\in \mathbb{Z},\
	+\infty=x_0>x_1>\ldots>x_N \}.
\end{equation}
By $\mathsf{P}^{\textnormal{TASEP}}_{q,\hmu,\hnu}$ denote
the $q$-Hahn TASEP Markov transition operator
acting on functions on $\mathbb{X}^N$.
Note that by the very definition of the $q$-Hahn TASEP,
the evolution of its $N$ rightmost particles is
independent from the rest of the process.

Also define the
Boson particle spaces
\begin{equation}\label{space_Y_N_y_Nk}
	\mathbb{Y}^N:=
	\{\vec y=(y_0,y_1,\ldots,y_N )\in\mathbb{Z}_{\ge0}^{N+1}\}
	,
	\qquad
	\mathbb{Y}^N_k:=
	\{\vec y\in \mathbb{Y}^{N}\colon y_0+y_1+\ldots+y_N=k \}.
\end{equation}
Let us define the
$q$-Hahn Boson operator $\mathsf{P}^{\textnormal{Boson}}_{q,\alpha,\hnu}$
acting on functions on $\mathbb{Y}^{N}$,
where the parameter $\alpha$ is arbitrary.
First, let $[\mathcal{A}_{q,\alpha,\hnu}]_i$, $i=1,\ldots,N $,
be a local operator acting only on the $y_i$ and $y_{i-1}$ coordinates of functions $f$  as
\begin{equation*}
	[\mathcal{A}_{q,\alpha,\hnu}]_i
	f(\vec y)
	=
	\sum_{s_i=0}^{y_i}
	\varphi_{q,\alpha,\hnu}(s_i\mid y_i)\,
	f(y_0,y_1,\ldots,y_{i-1}+s_i,y_i-s_i,\ldots,y_N).
\end{equation*}
Define the $q$-Hahn Boson  operator by its action
\begin{equation*}
	\mathsf{P}^{\textnormal{Boson}}_{q,\alpha,\hnu}
	f(\vec y)
	:=
	[\mathcal{A}_{q,\alpha,\hnu}]_N
	\ldots
	[\mathcal{A}_{q,\alpha,\hnu}]_1
	f(\vec y).
\end{equation*}
The order of operators is important and corresponds
to parallel update: first move particles from location~1 to location~0, 
then from location~2 to location~1 (note that the distribution of the particles
taken from location~2 is no affected by what happened at location~1),
and so on.
No particles are moved from the location $0$, and no
particles are added to the location $N$.
This implies that
$\mathsf{P}^{\textnormal{Boson}}_{q,\alpha,\hnu}$
preserves each of the spaces $\mathbb{Y}^{N}_k$.

\begin{remark}
	We put no restrictions on the parameter
	$\alpha$ in
	$\mathsf{P}^{\textnormal{Boson}}_{q,\alpha,\hnu}$.
	In particular, we do not require it to be a Markov transition operator: it is
	allowed to have negative matrix elements.
	However, the rows in
	$\mathsf{P}^{\textnormal{Boson}}_{q,\alpha,\hnu}$
	still sum to one,
	and for
	$\hnu\le \alpha<1$ the operator
	$\mathsf{P}^{\textnormal{Boson}}_{q,\alpha,\hnu}$
	defines a discrete time Markov
	process on $\mathbb{Y}^{N}$
	called the $q$-Hahn Boson system.
\end{remark}

The $q$-Hahn TASEP and $q$-Hahn Boson operators with the same
parameters $\alpha=\hmu$
are dual to each
other via the duality functional
$\mathfrak{H}\colon \mathbb{X}^{N}\times \mathbb{Y}^{N}\to\mathbb{R}$ defined as
\begin{equation}\label{H_duality_functional}
	\mathfrak{H}(\vec x,\vec y):=
	\begin{cases}
		\displaystyle\prod_{i=1}^{N}q^{y_i(x_i+i)},&y_0=0;
		\\[10pt]
		0,&y_0>0.
	\end{cases}
\end{equation}

\begin{theorem}[Duality for the $q$-Hahn TASEP \cite{Corwin2014qmunu}]
	\label{thm:TASEP_duality}
	We have
	\begin{equation*}
		\mathsf{P}^{\textnormal{TASEP}}_{q,\hmu,\hnu}\mathfrak{H}
		=
		\mathfrak{H}\bigl(\mathsf{P}^{\textnormal{Boson}}_{q,\hmu,\hnu}\bigr)^T,
	\end{equation*}
	where ``$T$'' stands for transpose of an operator.
	In other words,
	$\mathsf{P}^{\textnormal{TASEP}}_{q,\hmu,\hnu}$
	acts in the variables $\vec x$,
	and
	$\mathsf{P}^{\textnormal{Boson}}_{q,\hmu,\hnu}$
	acts in the variables
	$\vec y$, and their actions on
	$\mathfrak{H}(\vec x,\vec y)$
	coincide.
\end{theorem}

An immediate consequence of the duality is that the $q$-moments
$\mathop{\mathbb{E}}\bigl[\prod_{i=1}^{n}q^{y_i(x_i(t)+i)}\bigr]$
of the $q$-Hahn TASEP, as index by time $t$ and the vector $\vec{y}$ solves a difference equation involving the $q$-Hahn Boson operator. Using the Bethe ansatz solvability of this Boson operator allows one to write down explicit
contour integral formulas
for the $q$-moments, which become particularly nice when the $q$-Hahn TASEP is started from the step
\cite{Corwin2014qmunu}
or the
half-stationary
\cite{BCPS2014}
initial data.

Let us recall the formulas in the step case.
Encode elements of $\mathbb{Y}^N_k$ as
$\vec n=(n_1\ge \ldots\ge n_k)$, $N\ge n_1$, $n_k\ge0$,
where for all $m$ we have
$y_m=\#\{i\colon n_i=m\}$.\footnote{For example, $\vec y=(1,0,3,1,2)\in \mathbb{Y}^4_7$
corresponds to $\vec n=(4,4,3,2,2,2,0)$.}
\begin{theorem}[{\cite[Theorem 1.9]{Corwin2014qmunu}}]
	\label{thm:TASEP_moments}
	Fix $0<q<1$, $0\le \hnu\le\hmu<1$.
	For any $N,k\ge1$ and $\vec n$ as above, the
	$q$-moments of the $q$-Hahn PushTASEP with step initial
	data $x_i(0)=-i$, $i\ge1$, have the form
	\begin{multline*}
		\mathop{\mathbb{E}}
		\Bigl[
			\prod_{i=1}^{k}q^{x_{n_i}(t)+n_i}
		\Bigr]
		=\frac{(-1)^kq^{\frac{k(k-1)}{2}}}{(2\pi\sqrt{-1})^k}
		\oint \frac{dz_1}{z_1}\ldots
		\oint \frac{dz_k}{z_k}
		\prod_{1\le A<B\le k}\frac{z_A-z_B}{z_A-qz_B}
		\\\times\prod_{i=1}^{k}
		\Biggl[
			\left( \frac{1-\hnu z_j}{1-z_j} \right)^{n_j}
			\left( \frac{1-\hmu z_j}{1-\hnu z_j} \right)^t
			\frac{1}{1-\hnu z_j}
		\Biggr].
	\end{multline*}
	Here all the integration contours encircle $1$ but not $0$ or $\hnu^{-1}$, and for all $B>A$
	the $z_A$ contour encircles the $qz_B$ contour.
\end{theorem}
The parameters $y_i$ are the labels of the $q$-moments.
In \Cref{sub:duality}
we prove a duality result for the $q$-Hahn PushTASEP,
and in
\Cref{sub:push_contour_integrals}
utilize the duality to obtain contour integral
formulas for the $q$-moments of this process.

\section{$q$-Hahn PushTASEP, duality, and contour integrals}
\label{sec:q_hahn_push}

\subsection{Definition and nonnegativity}
\label{sub:def_nonnegativity}

The $q$-Hahn PushTASEP depends on the main
``quantization'' parameter $q\in(0,1)$, and
on two parameters $\mu,\nu$ in the following range:
\begin{equation}
	\label{mu_nu_parameters_for_qHahn}
	0<\mu<1,\qquad -1<\nu\le \min\{\mu,\sqrt q\}.
\end{equation}
The $q$-Hahn PushTASEP is a discrete time Markov
process on particle configurations in $\mathbb{Z}$
(with at most one particle per site)
which have a rightmost particle:
\begin{equation*}
	\vec x(t)=(x_1(t)> x_2(t)> \ldots).
\end{equation*}
At each discrete time moment,
particles in the $q$-Hahn PushTASEP may jump to the left.
The update $\vec x(t)\to \vec x(t+1)$
is performed according to the following procedure
(the distributions $\varphi$ and $\psi$ are defined
	by \eqref{phi_qmunu_definition}
	and \eqref{q_hyp_distribution_definition},
	respectively):
	\begin{enumerate}[\bf1.\/]
	\item The first particle $x_1$ jumps to the left by $\ell\in \mathbb{Z}_{\ge0}$, where
		$\ell$ is drawn from the distribution $\varphi_{q,\mu,\nu}(\ell\mid \infty)$.
	\item Consecutively for $i=2,3,\ldots $, given the movement of the
		$(i-1)$-st particle $x_{i-1}(t)\to x_{i-1}(t+1)=x_{i-1}(t)-\ell$,
		and the gap $g=x_{i-1}(t)-x_{i}(t)-1$ before this movement, the
		location of the
		$i$-th particle is updated as
		$x_i(t)\to x_{i}(t+1)=x_{i}(t)-L$, $L\in \mathbb{Z}_{\ge0}$, with probability
		\begin{equation}
			\label{P_ell_g_update_probability_definition}
			\mathbf{P}_{\ell, g}(L):=
			\sum_{p=0}^{\min\{\ell, L\}}
			\varphi_{q^{-1}, q^{g}, \mu \nu q^{g-1}}(p \mid \ell)
			\,
			\psi_{q, \nu \mu^{-1} q^{p}, \nu q^{g}, \nu^{2} q^{g+p}}(L-p).
		\end{equation}
		For consistency of notation we will sometimes write
		$x_0=+\infty$ and $\mathbf{P}_{\ell,\infty}(L)=\varphi_{q,\mu,\nu}(L\mid \infty)$.
		See \Cref{fig:push_update_pic} for an illustration.
\end{enumerate}
\begin{figure}[htbp]
	\centering
	\includegraphics[width=.7\textwidth]{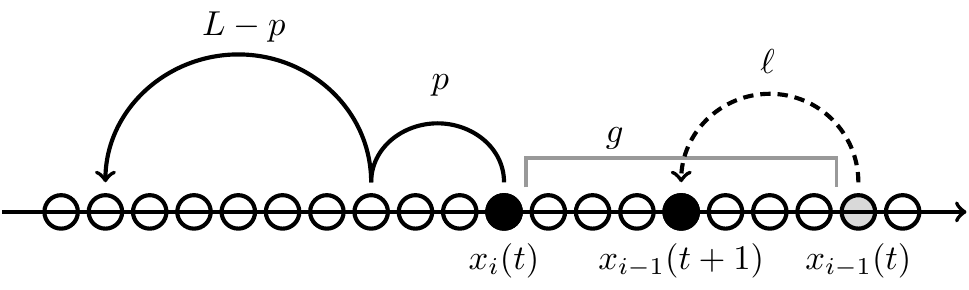}
	\caption{Update of the $i$-th particle given the movement of the $(i-1)$-st particle.}
	\label{fig:push_update_pic}
\end{figure}

Let us make a number of comments concerning this definition:
\begin{enumerate}[$\star$]
	\item It is not obvious that the right-hand side of \eqref{P_ell_g_update_probability_definition}
		is
		nonnegative because for $\nu>0$ the expression
		$\varphi_{q^{-1}, q^{g}, \mu \nu q^{g-1}}(p \mid \ell)$ might be negative.
		We prove the nonnegativity of the $\mathbf{P}_{\ell,g}$'s
		in
		\Cref{prop:nonnegativity} below.
	\item The update probabilities $\mathbf{P}_{\ell,g}$
		are given by a complicated expression involving $q$-Pochhammer symbols.
		In \Cref{sub:degenerations} below we discuss a number of previously studied
		PushTASEP like processes
		which arise as degenerations of the $q$-Hahn PushTASEP.
		In particular, under these degenerations the
		$\mathbf{P}_{\ell,g}$'s simplify.
	\item
		One can check that $\mathbf{P}_{\ell,g}(L)=0$ unless $L\ge \ell-g$.
		(Indeed, this is because $\varphi_{q^{-1}, q^{g}, \mu \nu q^{g-1}}(p \mid \ell)=0$
		unless $\ell-p\le g$.) 
		In words, if the previous jumping distance $\ell$
		is greater than the gap between $x_{i-1}$ and $x_i$,
		then the particle $x_i$ is deterministically pushed to the left.
		Therefore,
		update rule \eqref{P_ell_g_update_probability_definition}
		preserves the order of the particles.
	\item
		If $\ell>g$ and $\mu\nu$ equals $q^y$ for a positive integer $y$, the
		denominator in
		$\varphi_{q^{-1}, q^{g}, \mu \nu q^{g-1}}(p \mid \ell)$
		\eqref{phi_qmunu_definition}
		may vanish. However, we can still define $\varphi$ by continuity (canceling the corresponding factor in the numerator).
	\item
		The
		process can make infinitely many
		jumps in a single discrete time step
		(for example,
		when it starts from the step initial configuration
		$x_i(0)=-i$, $i\ge0$).
		However, for each $N$
		the behavior of the particles $x_i$, $i\le N$, is independent from the one of
		with $i>N$. Therefore, the dynamics restricted to $x_1,\ldots,x_N$
		is well-defined as a process with finitely many particles.
		These $N$-particle dynamics are compatible for various $N$, and so
		the dynamics of the process on infinite particle configurations
		having a rightmost particle is well-defined.
\end{enumerate}

\begin{proposition}
	\label{prop:nonnegativity}
	Under the restrictions \eqref{mu_nu_parameters_for_qHahn}
	on parameters
	we have $\mathbf{P}_{\ell,g}(L)\ge0$
	for all $\ell,g,L\in \mathbb{Z}_{\ge0}$, and
	$\sum_{L\in \mathbb{Z}_{\ge0}}\mathbf{P}_{\ell,g}(L) = 1$.
\end{proposition}
\begin{proof}
	The fact that
	$\sum_{L\in \mathbb{Z}_{\ge0}}\mathbf{P}_{\ell,g}(L) = 1$
	follows by interchanging the summations over $L$ and $p$
	(the latter coming from \eqref{P_ell_g_update_probability_definition})
	and using
	\begin{equation*}
		\sum_{p=0}^{\ell}\varphi_{q^{-1}, q^{g}, \mu \nu q^{g-1}}(p \mid \ell)=1,
		\qquad
		\sum_{L=p}^{\infty}\psi_{q, \nu \mu^{-1} q^{p}, \nu q^{g}, \nu^{2} q^{g+p}}(L-p)=1,
	\end{equation*}
	see \Cref{sec:q_hyp}.

	Turning to proving the positivity, there are two cases depending on the sign of $\nu$.
	First,
	we have
	$\varphi_{q^{-1}, q^{g}, \mu \nu q^{g-1}}(p \mid \ell) \geq 0$
	when $-1<\nu\le0$.
	Therefore,
	we can think that the $i$-th particle first moves $x_{i}(t) \to x_{i}(t)-p$
	with probability
	$\varphi_{q^{-1}, q^{g}, \mu \nu q^{g-1}}(p \mid \ell)$
	due to the push of the $i-1$-st particle. After that, the $i$-th particle makes an
	extra move $x_{i}(t)-p \to x_{i}(t)-p-m$ (where $m+p=L$)
	with probability $\psi_{q, \nu \mu^{-1} q^{p}, \nu q^{g}, \nu^{2} q^{g+p}}(m) \geq 0$.
	In other words, the jump by $L$
	in \Cref{fig:push_update_pic}
	is a combination of two jumps, by $p$ and $L-p$ (with $p$ random),
	each happening with a nonnegative probability.

	In the second case $0<\nu\le\min\{\mu,\sqrt q\}$
	the expression
	$\varphi_{q^{-1}, q^{g}, \mu \nu q^{g-1}}(p \mid \ell)$ might become negative,
	and the previous interpretation does not imply nonnegativity of $\mathbf{P}_{\ell,g}(L)$.
	Let us rewrite $\mathbf{P}_{\ell,g}(L)$
	in two different ways (depending on the order of $\ell$ and $g$)
	to show the nonnegativity.
	We have (here and below in the proof we use the notation from \Cref{sub:q_hyp_notation})
	\begin{align}
		\nonumber
		\mathbf{P}_{\ell,g}(L)
		&=
		\sum_{p=0}^{\min\{\ell,L\}}
		q^{gp}\frac{(\mu \nu/q;q^{-1})_{p}(q^{g};q^{-1})_{\ell-p}}{(\mu\nu q^{g-1};q^{-1})_{\ell}}
		\frac{(q^{-1};q^{-1})_{\ell}}{(q^{-1};q^{-1})_{p}(q^{-1};q^{-1})_{\ell-p}}
		\\&
		\nonumber
		\hspace{100pt}\times
		\mu^{L-p}
		\frac{(\nu\mu^{-1}q^p,\nu q^{g};q)_{L-p}}{(q,\nu^2 q^{g+p};q)_{L-p}}
	\frac{(\nu^2 q^{g+p},\mu;q)_{\infty}}{(\nu q^p,\mu\nu q^g;q)_{\infty}}
		\\&
		\label{P_ell_g_initial_formula}
		=
		\frac{(\mu; q)_{\infty}(\nu^{2} q^{g}; q)_{\infty}}{(\nu; q)_{\infty}(\nu \mu q^{g}; q)_{\infty}}
		\sum_{p=0}^{\min\{\ell,L\}}
		q^{gp}\frac{(\mu \nu/q;q^{-1})_{p}(q^{g};q^{-1})_{\ell-p}}{(\mu\nu q^{g-1};q^{-1})_{\ell}}
		\frac{(q^{-1};q^{-1})_{\ell}}{(q^{-1};q^{-1})_{p}(q^{-1};q^{-1})_{\ell-p}}
		\\&
		\hspace{100pt}\times
		\mu^{L-p}
		\frac{(\nu\mu^{-1}q^p,\nu q^{g};q)_{L-p}}{(q,\nu^2 q^{g+p};q)_{L-p}}
		\frac{(\nu;q)_p}{(\nu^2 q^g;q)_p}.
		\nonumber
	\end{align}
	When
	$\ell<g$, this expression is rewritten as follows:
	\begin{multline}
		\label{P_ell_g_rewrite_first_case}
		\mathbf{P}_{\ell,g}(L)
		=
		\mu^{L}\frac{(q^{g}; q^{-1})_{\ell} (\nu \mu^{-1};
		q)_{L}(\nu q^{g}; q)_{L}}{(\mu \nu q^{g-1}; q^{-1})_{\ell} (q;
		q)_{L}(\nu^{2} q^{g}; q)_{L}}
		\\
		{_{4}}\phi_{3}
		\left[\begin{array}{cccc} q^{-\ell} & \mu^{-1} \nu^{-1} q &  \nu &  q^{-L}
		\\  & \nu^{-1}q^{1-g-L} &  \nu \mu^{-1} & q^{g-\ell+1} \end{array}; q, q
		\right]
		\frac{(\mu; q)_{\infty}(\nu^{2} q^{g}; q)_{\infty}}{(\nu;
		q)_{\infty}(\nu \mu q^{g}; q)_{\infty}}.
	\end{multline}
	The equality between \eqref{P_ell_g_initial_formula} and
	\eqref{P_ell_g_rewrite_first_case} is termwise (when using the definition
	\eqref{q_hyp_defn} for $_4\phi_3$).

	We will use
	Watson's transformation formula
	\cite[(III.19)]{GasperRahman},
	\begin{multline}
		\label{Watsontransformation}
		{_{4}}\phi_{3} \left[
			\begin{array}{cccc}
				q^{-n} & a & b & c  \\  & d & e & f
			\end{array}
			; q, q \right]
		\\=
		\frac{(d/b, d/c; q)_{n}}{(d, d/bc; q)_{n}}\,
		{_{8}}\phi_{7} \left[
			\begin{array}{cccccccc} q^{-n} & \sigma & q \sigma^{\frac{1}{2}} & -q \sigma^{\frac{1}{2}} & f/a & e/a & b & c
			\\
			& \sigma^{\frac{1}{2}} &  - \sigma^{\frac{1}{2}} & e & f & ef/ab & ef/ac & e f q^{n}/a
		\end{array}; q, \frac{e f q^{n}}{bc} \right],
	\end{multline}
	where $d e f= a b c q^{1-n}$ and $\sigma = ef/aq$.
	Applying this formula to \eqref{P_ell_g_rewrite_first_case},
	we obtain
	\begin{multline}
		\label{P_ell_g_rewrite_first_case_87}
		\mathbf{P}_{\ell,g}(L)
		=
		\frac
		{(\mu; q)_{\infty}(\nu^{2} q^{g}; q)_{\infty}}
		{(\nu; q)_{\infty}(\nu \mu q^{g}; q)_{\infty}}
		\,
		\mu^{L}
		\frac{(q^{g}; q^{-1})_{\ell} (\nu \mu^{-1}; q)_{L}(\nu q^{g}; q)_{L}(\nu^{-1}q^{1-g}, q)_{\ell}(\nu^{-2} q^{1-g-L}; q)_{\ell}}
		{(\mu \nu q^{g-1}; q^{-1})_{\ell} (q; q)_{L}(\nu^{2} q^{g}; q)_{L}(\nu^{-1} q^{1-g-L}; q)_{\ell}(\nu^{-2}q^{1-g}; q)_{\ell}}
		\\ \times
		{_{8}}\phi_{7} \left[\begin{array}{cccccccc}
				q^{-\ell} & \nu^{2} q^{g-\ell-1} & \nu q^{\frac{g-\ell+1}2} & -\nu q^{\frac{g-\ell+1}2} & \nu^{2}/q & \mu \nu q^{g-\ell} & \nu & q^{-L}
			\\
			& \nu q^{\frac{g-\ell-1}2} &  - \nu q^{\frac{g-\ell-1}2} & q^{g-\ell+1} & \nu/\mu  & \nu q^{g-\ell} & \nu^{2}q^{L+g-\ell} & \nu^{2}q^{g}
		\end{array};
		q, \frac{q^{L+g+1}}{\mu} \right].
	\end{multline}

	When $\ell\ge g$, we must have $L\ge \ell-g$. Let us rewrite \eqref{P_ell_g_initial_formula} in another form:
	\begin{multline}
		\label{P_ell_g_rewrite_second_case}
		\mathbf{P}_{\ell,g}(L)
		=
		\mu^{L-\ell+g}
		\frac
		{(q^{\ell}; q^{-1})_{g} (\nu \mu^{-1} q^{\ell-g};q)_{L-\ell+g} (\nu; q)_{\ell-g}(\nu q^{g}; q)_{L-\ell+g}}
		{(\mu \nu; q)_{g} (q; q)_{L-\ell+g}(\nu^{2} q^{g}; q)_{L}}
		\\ \times
		{_{4}}\phi_{3}
		\left[\begin{array}{cccc}
				q^{-g} & \mu^{-1} \nu^{-1} q^{\ell-g+1} &  \nu
				q^{\ell-g} &  q^{-L+\ell-g}
				\\
				& \nu^{-1}q^{-L+\ell-2g+1} &  \nu \mu^{-1}
				q^{\ell-g} & q^{\ell-g+1}
			\end{array}; q, q
		\right]
		\frac{(\mu;
		q)_{\infty}(\nu^{2} q^{g}; q)_{\infty}}{(\nu; q)_{\infty}(\nu \mu q^{g};
		q)_{\infty}}.
	\end{multline}
	Again, the equality between \eqref{P_ell_g_initial_formula} and
	\eqref{P_ell_g_rewrite_second_case} is termwise up to an index shift by $\ell-g$.
	Using Watson's transformation formula
	\eqref{Watsontransformation}, we get
	\begin{align}
		\nonumber&
		\mathbf{P}_{\ell,g}(L)
		=
		\frac
		{(\mu; q)_{\infty}(\nu^{2} q^{g}; q)_{\infty}}
		{(\nu; q)_{\infty}(\nu \mu q^{g}; q)_{\infty}}
		\\
		\label{P_ell_g_rewrite_second_case_87}
		&\hspace{10pt}\times
		\mu^{L-\ell+g}\,
		\frac{(q^{\ell}; q^{-1})_{g} (\nu \mu^{-1} q^{\ell-g};q)_{L-\ell+g} (\nu; q)_{\ell-g}
		(\nu q^{g};q)_{L-\ell+g}(\nu^{-2}q^{-L-g+1}; q)_{g} (\nu^{-1}q^{-g+1}; q)_{g}}
		{(\mu \nu;q)_{g} (q; q)_{L-\ell+g}(\nu^{2} q^{g}; q)_{L} (\nu^{-1} q^{-L+\ell-2g+1};q)_{g}
		(\nu^{-2}q^{-\ell+1}; q)_{g}}
		\\ &\hspace{10pt}
		\nonumber
		\times {_{8}}\phi_{7}
		\left[\begin{array}{cccccccc}
				q^{-g} & \nu^{2} q^{\ell-g-1} & \nu
				q^{\frac{\ell-g+1}2} & -\nu q^{\frac{\ell-g+1}2} & \nu^{2}/q & \mu \nu & \nu
				q^{\ell-g} & q^{-L+\ell-g}
				\\
				& \nu q^{\frac{\ell-g-1}2} &  - \nu q^{\frac{\ell-g-1}2}
				& q^{\ell-g+1} & \nu/ \mu q^{\ell-g} & \nu & \nu^{2}q^{L} & \nu^{2}q^{\ell}
			\end{array}; q, \frac{q^{L+g+1}}{\mu}
		\right].
	\end{align}
	In both cases
	\eqref{P_ell_g_rewrite_first_case_87} and \eqref{P_ell_g_rewrite_second_case_87}
	all the
	prefactors and the
	terms in the sums for
	$_8\phi_7$ are manifestly nonnegative under our conditions \eqref{mu_nu_parameters_for_qHahn}
	(for the terms in $_8\phi_7$ this follows from the definition of $_8\phi_7$ as a sum).
	This completes the proof. 
\end{proof}
\begin{remark}
	\label{rmk:sqrt_q_maybe_a_proof_artefact}
	The condition $\nu\le \sqrt{q}$ in \eqref{mu_nu_parameters_for_qHahn},
	which was not present in the $q$-Hahn TASEP
	\cite{Corwin2014qmunu}, \cite{CorwinPetrov2015},
	is essential for the nonnegativity of transition probabilities in our $q$-Hahn PushTASEP. Indeed,
	\begin{equation*}
		\mathbf{P}_{1,1}(1)=
		\frac{\mu  \left(\nu ^2-\nu +1\right)-\nu +q \left( 1+\nu ^2-(\mu +1) \nu\right)}{(1-\mu  \nu ) (1-q \nu ^2)},
	\end{equation*}
	and the numerator in this expression is negative for, say, $q=\frac{1}{4}$, $\mu=\frac{3}{4}$, and $\nu=\frac{2}{3}$.
\end{remark}

\subsection{Degenerations}
\label{sub:degenerations}

The $q$-Hahn PushTASEP
update probabilities $\mathbf{P}_{\ell,g}(L)$
are defined by rather complicated expressions
\eqref{P_ell_g_update_probability_definition}.
Here we discuss a number of their degenerations when the parameter $q=0$. These
lead to some known and some new
stochastic particles systems with pushing. In \Cref{sec:beta_limit} we discuss another type of degeneration where $q\to 1$ which also simplifies the form of the update probabilities.

\subsubsection{Known PushTASEPs}

If we set $\nu=0$, the
factor
$\psi_{q, \nu \mu^{-1} q^{p}, \nu q^{g}, \nu^{2} q^{g+p}}(L-p)$
in \eqref{P_ell_g_update_probability_definition}
simplifies to
\begin{equation*}
	\psi_{q, \nu \mu^{-1} q^{p}, \nu q^{g}, \nu^{2} q^{g+p}}(L-p)
	\,\big\vert_{\nu=0}=\mu^{L-p}\frac{(\mu;q)_{\infty}}{(q;q)_{L-p}}.
\end{equation*}
This is the $q$-geometric distribution.
The first particle jumps according
to this
$q$-geometric distribution (note that $p=0$ for the first particle).
For the update probabilities of all other particles
we have
\begin{equation}\label{deg_P_nu0}
	\mathbf{P}_{\ell,g}(L)\,\big\vert_{\nu=0}
	=
	(\mu;q)_{\infty}
	\sum_{p=0}^{\min\left\{ \ell,L \right\}}
	q^{gp}\mu^{L-p}
	\frac{(q^g;q^{-1})_{\ell-p}(q^{-1};q^{-1})_{\ell}}
	{(q;q)_{L-p}(q^{-1};q^{-1})_{p}(q^{-1};q^{-1})_{\ell-p}}.
\end{equation}
We see that the $\nu=0$ process coincides with the
geometric $q$-PushTASEP introduced in
\cite[Section 6.3]{MatveevPetrov2014},
and our parameter $\mu$
corresponds to
$\alpha a_i$ (specific to each particle).

\medskip

Further setting
$q=0$ in the geometric $q$-PushTASEP reduces it to a
geometric PushTASEP (that is, a discrete time PushTASEP
with geometrically distributed jumps).
In the sum \eqref{deg_P_nu0} only one summand
will be nonzero, and\footnote{Throughout the paper
$\mathbf{1}_{\cdots}$ denotes the indicator.}
\begin{equation}\label{geom_push_q0_nu0}
	\mathbf{P}_{\ell,g}(L)\,\big\vert_{\nu=q=0}
	=
	\sum_{p=0}^{\min\left\{ \ell,L \right\}}
	(1-\mu)\,\mu^{L-p}\,\mathbf{1}_{p=\max\left\{ 0,\ell-g \right\}}
	=
	(1-\mu)\,\mu^{L-\max\{0,\ell-g\}}.
\end{equation}
In words, the first
particle jumps according to a geometric
distribution with parameter $\mu$,
and for each $i=2,3,\ldots$,
the particle $x_{i-1}$ pushes $x_i$
by the minimal possible distance which
preserves the order of the particles
(cf.~\Cref{fig:push_update_pic}),
and after this push the particle $x_i$ makes an independent
jump according to the geometric distribution.
The geometric PushTASEP is well-known,
see, e.g.,
\cite{BorFerr2008DF},
\cite{warrenwindridge2009some}
for its connections to Schur processes and
dynamics on them.

\medskip

Finally, in both the geometric
$q$-PushTASEP and the geometric PushTASEP
one can pass to
the continuous time by sending $\mu\to0$ and
rescaling the discrete time by $\mu^{-1}$.
For $q>0$, this produces the continuous time
$q$-PushTASEP
introduced in
\cite{BorodinPetrov2013NN} and considered in \cite{CorwinPetrov2015}.
For $q=0$ the process reduces
to the usual continuous time PushTASEP
\cite{Spitzer1970},
\cite{derrida1991dynamics}
(also referred to as the ``long-range TASEP'', and viewed as a
simplified model of Toom's anchored interface dynamics).

\begin{remark}[Particle-dependent parameters]
	In the geometric $q$-PushTASEP
	from \cite{MatveevPetrov2014},
	as well as in its degenerations,
	one can assign update probability
	with the particle-dependent parameter $\mu_i$
	to each $x_i$, and the
	resulting system remains exactly solvable via
	$q$-Whittaker
	or Schur symmetric functions
	(for the $q$-PushTASEP and PushTASEPs, respectively).
	Moreover, in the $q$-Hahn TASEP one can also take particle-dependent
	parameters $(\mu_i,\nu_i)$, under the condition that $\nu_i/\mu_i$ does not
	depend on the particle. The $q$-moments of the resulting system
	can be expressed as contour integrals
	coming from the inhomogeneous stochastic higher
	spin six vertex model
	\cite[Section 10.3]{BorodinPetrov2016inhom}.
	It is likely that the $q$-Hahn PushTASEP duality and
	moment formulas can be extended
	to include particle-dependent parameters,
	but here we do not pursue this direction.\footnote{The upcoming work
	\cite{BufetovMucciconiPetrov2018}
	connects the $q$-Hahn PushTASEP to stochastic vertex models.
	Formulas for observables in the $q$-Hahn PushTASEP
	with particle-dependent parameters
	can likely be obtained
	using that approach.}
\end{remark}

\subsubsection{A new $q=0$ particle system}
\label{sub:gB_push}

Let us briefly describe a limit as $q\searrow0$ and $\mu,\nu$ are fixed,
which leads to a new particle system.
Define the geometric-Bernoulli probability
distribution\footnote{This random
variable is a product of a Bernoulli random variable
with values in $\left\{ 0,1 \right\}$ and an independent
geometric random variable with values in $\left\{ 1,2,\ldots  \right\}$,
hence the
name.}
on $k\in \mathbb{Z}_{\ge0}$
by
\begin{equation*}
	\mathsf{gB}(\alpha,\beta;k)
	:=
	\begin{cases}
		\beta,& k=0;\\
		(1-\beta)(1-\alpha)\alpha^{k-1},& k\ge1.
	\end{cases}
\end{equation*}
We have for the two factors in the
sum in
\eqref{P_ell_g_update_probability_definition}:
\begin{equation}\label{q0_deg}
	\varphi_{q^{-1}, q^{g}, \mu \nu q^{g-1}}(p \mid \ell)
	\,\Big\vert_{q=0}
	=
	\begin{cases}
		1,&\textnormal{if $g=0$ and $p=\ell$};\\
		\dfrac{1}{1-\mu\nu},&
		\textnormal{if $0<g\le \ell$ and $p=\ell-g$};\\[9pt]
		\dfrac{-\mu\nu}{1-\mu\nu},&
		\textnormal{if $0<g\le\ell$ and $p=\ell-g+1$};\\[6pt]
		1,&\textnormal{if $g>\ell$ and $p=0$};\\
		0,&\textnormal{otherwise},
	\end{cases}
\end{equation}
and
\begin{equation}\label{q0_deg1}
	\psi_{q, \nu \mu^{-1} q^{p}, \nu q^{g}, \nu^{2} q^{g+p}}(L-p)
	\,\Big\vert_{q=0}
	=
	\begin{cases}
		\mathsf{gB}(\mu,\frac{(1-\mu)(1+\nu)}{1-\mu\nu};L),
		& p=g=0;\\[6pt]
		\mathsf{gB}(\mu,\frac{1-\mu}{1-\nu};L),
		&p=0,\ g>0;\\[6pt]
		\mathsf{gB}(\mu,\frac{1-\mu}{1-\mu\nu};L-p),
		& p>0,\ g=0;\\[6pt]
		\mathsf{gB}(\mu,1-\mu;L-p),
		& p>0,\ g>0.
	\end{cases}
\end{equation}
The first particle in this $q=0$ process
jumps according to the distribution
$\mathsf{gB}(\mu,\frac{1-\mu}{1-\nu};\cdot)$.
To write down the update probabilities
of all other particles, observe that
the combination \eqref{P_ell_g_update_probability_definition}
of the quantities \eqref{q0_deg}--\eqref{q0_deg1}
takes the form
(recall that $\mathbf{P}_{\ell,g}(L)=0$ unless $L\ge \ell-g$)
\begin{equation}\label{q0_full_prob_P}
	\mathbf{P}_{\ell,g}(L)\,\big\vert_{q=0}=
	\begin{cases}
		\mathsf{gB}(\mu,\frac{1-\mu}{1-\nu};L)
		& \ell<g;
		\\[6pt]
		\mathsf{gB}(\mu,\frac{1-\mu}{1-\mu\nu};L-\ell+g)
		& \ell>g,\ L\ge \ell-g;
		\\[6pt]
		\mathsf{gB}(\mu,\frac{1-\mu}{(1-\nu)(1-\mu\nu)};L)
		& \ell=g>0;
		\\[6pt]
		\mathsf{gB}(\mu,\frac{(1-\mu)(1+\nu)}{1-\mu\nu};L)
		& \ell=g=0.
	\end{cases}
\end{equation}
The condition that all the update probabilities
$\mathbf{P}_{\ell,g}(L)\,\vert_{q=0}$
are nonnegative is equivalent to
$0\le \mu<1$ and that
all the second parameters
of the geometric-Bernoulli distributions
in \eqref{q0_full_prob_P}
are between $0$ and $1$. This leads to
\begin{equation*}
	-1\le \nu\le 0, \ 0 \le \mu<1
	\qquad
	\textnormal{or}
	\qquad
	0<\nu<1,\ \frac{\nu}{1-\nu+\nu^2}\le \mu<1.
\end{equation*}
For $\nu\le 0$ we can interpret \eqref{q0_deg} as a
random pushing caused by the jump of $x_{i-1}$,
and \eqref{q0_deg1} as an independent jump of
$x_i$ after the push (cf. \Cref{fig:push_update_pic}).

Setting $\nu=0$ in \eqref{q0_full_prob_P}
turns all the geometric-Bernoulli probabilities
into the geometric ones with parameter $\mu$.
This
recovers the discrete time geometric PushTASEP
as in
\eqref{geom_push_q0_nu0}.
Thus, the $q=0$ degeneration of the $q$-Hahn PushTASEP
can be viewed as a new nontrivial one-parameter
extension of the geometric PushTASEP.
Because $q$-moment formulas do not easily
survive the $q\searrow0$ degeneration,
here we do not pursue computations for this $q=0$ particle system.

A $q=0$ degeneration of the $q$-Hahn TASEP
was introduced in
\cite{derbyshev2012totally},
\cite{povolotsky2015gen_tasep}.
Its asymptotic behavior was studied recently in
\cite{SaenzKnizelPetrov2018}
via Schur measures.

\subsection{Duality}
\label{sub:duality}

Like the
$q$-Hahn TASEP,
our
$q$-Hahn PushTASEP satisfies a duality relation which we now describe.
This is one of the main results of the present work.

Fix $N$ and $k$ and recall the spaces
$\mathbb{X}^N$, $\mathbb{Y}^N$,
and $\mathbb{Y}^{N}_k$
\eqref{space_X_N}--\eqref{space_Y_N_y_Nk}.
Let $\mathsf{P}^{\textnormal{PushTASEP}}_{q,\mu,\nu}$
denote the one-step Markov transition operator of the $q$-Hahn PushTASEP
acting on $\mathbb{X}^{N}$
(the evolution of the $N$ rightmost particles under the $q$-Hahn PushTASEP
is independent from the rest of the system).
Recall the duality functional
$\mathfrak{H}(\vec x,\vec y)=\prod_{i=0}^{N}q^{y_i(x_i+i)}$
\eqref{H_duality_functional}
which, by agreement, is zero unless $y_0=0$.
Here $\vec x\in \mathbb{X}^N$, $\vec y\in \mathbb{Y}^{N}$.
By $\mathfrak{H}_k$ denote the restriction of $\mathfrak{H}$ to
$\mathbb{X}^N\times \mathbb{Y}^N_k$.
\begin{theorem}
	Let $(q,\mu,\nu)$ be the parameters of the $q$-Hahn PushTASEP
	satisfying
	\eqref{mu_nu_parameters_for_qHahn}. For any $k\ge 1$
	there exists $\mu_0>0$ (depending on $k$) such that for all $0<\mu<\mu_0$
	we have
	\label{thm:push_duality}
	\begin{equation}
		\label{push_duality_formulation_equation}
		\mu^k \mathsf{P}^{\textnormal{PushTASEP}}_{q,\mu,\nu} \mathfrak{H}_k
		\bigl(\mathsf{P}^{\textnormal{Boson}}_{q,q/\mu,\nu}\bigr)^{T}
		=
		\nu^k
		\mathfrak{H}_k
		\bigl(\mathsf{P}^{\textnormal{Boson}}_{q,q/\nu,\nu}\bigr)^{T}.
	\end{equation}
	Here ``$T$'' means transpose, that is, the Boson operators act in the
	$\vec y$ variables while the $q$-Hahn PushTASEP transition operator
	acts on the $\vec x$'s.
\end{theorem}
\begin{remark}
	The condition that $\mu$ is sufficiently small guarantees the convergence of the
	infinite
	series coming from the action of
	$\mathsf{P}^{\textnormal{PushTASEP}}_{q,\mu,\nu}$.
	This convergence issue
	is the reason that only finitely many of the $q$-moments
	of the $q$-Hahn PushTASEP exist and are given by
	the contour integrals (\Cref{thm:push_q_moments_in_the_text}, see also \Cref{lemma:finiteness_q_moments}).
	
	Note also that unlike in the $q$-Hahn TASEP
	duality (\Cref{thm:TASEP_duality}),
	in \eqref{push_duality_formulation_equation} the Boson operators do not necessarily
	have nonnegative matrix elements. This duality is {\it not} a Markov duality due to this lack of positivity as well as the factors $\mu^k$ and $\nu^k$. Despite this, we will see that it still provides meaningful information about how the expected value of the duality function evolves over time.
\end{remark}

The duality relation \eqref{push_duality_formulation_equation} was guessed from the contour
integral formulas (\Cref{thm:push_q_moments_in_the_text}) which generalize those known for the geometric $q$-pushTASEP (with step initial data). It was not a priori clear that the guessed formulas encoded expectations for any particle system. However, we discovered the $q$-Hahn pushTASEP introduced here satisfies both the duality (which is a result for general initial data) and the contour integral formulas (again, for step initial data).

Here we  directly verify that
the $q$-Hahn PushTASEP
defined in
\Cref{sub:def_nonnegativity}
satisfies
\eqref{push_duality_formulation_equation}.
The proof of the duality relation occupies the rest of this subsection
and is based on \Cref{lemma:main_identity} which we prove in the next
\Cref{sub:lemma_proof}.

First, let us write \eqref{push_duality_formulation_equation}
out for fixed $\vec x=(x_1,\ldots,x_N )\in \mathbb{X}^{N}$ and
$\vec y=(y_0,y_1,\ldots,y_N )\in \mathbb{Y}^{N}_{k}$.
We need to show that
\begin{equation}
	\label{push_duality_formulation_equation_as_sums}
	\mu^k
	\sum_{\vec x'\in \mathbb{X}^{N}}
	\sum_{\vec y'\in \mathbb{Y}^{N}_k}
	\mathsf{P}^{\textnormal{PushTASEP}}_{q,\mu,\nu}(\vec x,\vec x')
	\mathfrak{H}_k (\vec x',\vec y')
	\mathsf{P}^{\textnormal{Boson}}_{q,q/\mu,\nu}(\vec y,\vec y')
	=
	\nu^k
	\sum_{\vec y''\in \mathbb{Y}^{N}_k}
	\mathfrak{H}_k (\vec x,\vec y'')
	\mathsf{P}^{\textnormal{Boson}}_{q,q/\nu,\nu}(\vec y,\vec y'').
\end{equation}
If $y_0>0$, both sides of \eqref{push_duality_formulation_equation_as_sums}
vanish because
$\mathsf{P}^{\textnormal{Boson}}_{q,q/\nu,\nu}(\vec y,\vec y')$
is nonzero only when $y_0'\ge y_0$, and we use the definition of $\mathfrak{H}$
\eqref{H_duality_functional}.
Thus, we can and will assume that $y_0=0$.
Continuing, we further expand \eqref{push_duality_formulation_equation_as_sums}:
\begin{equation}
	\label{push_duality_formulation_equation_detail}
	\begin{split}
		&
		\mu^k
		\sum_{\vec x'\in \mathbb{X}^{N}}
		\sum_{\vec y'\in \mathbb{Y}^{N}_k}
		\,\prod_{i=1}^N
		\mathbf{P}_{x_{i-1}-x_{i-1}',x_{i-1}-x_i-1}(x_i-x_i')
		\prod_{i=1}^{N}
		\varphi_{q,q/\mu,\nu}
		\left( \sum\nolimits_{j=0}^{i-1}(y_j'-y_j)\;\Big\vert\;y_i \right)
		\prod_{r=0}^{N}q^{y_r'(x_r'+r)}
		\\
		&\hspace{120pt}
		=
		\nu^k
		\sum_{\vec y''\in \mathbb{Y}^{N}_k}
		\,
		\prod_{i=1}^{N}\varphi_{q,q/\nu,\nu}
		\left( \sum\nolimits_{j=0}^{i-1}(y_j''-y_j)\;\Big\vert\;y_i \right)
		\prod_{r=0}^{N}q^{y_r''(x_r+r)}
	\end{split}
\end{equation}
(the products over $r$ in both sides vanish if $y_0'$ or $y_0''$,
respectively, are positive).
We will prove \eqref{push_duality_formulation_equation_detail}
by induction on $N$.
\begin{lemma}[Induction base, case $N=1$]
	\label{lemma:induciton_base}
	For any $x_1\in \mathbb{Z}$ and $y_1\ge0$
	there exists $\mu_0>0$ such that
	for all $0<\mu<\mu_0$
	we have
	\begin{equation*}
		\mu^{y_1}\varphi_{q,q/\mu,\nu}(0\mid y_1)
		\sum_{x_1'=-\infty}^{x_1}\varphi_{q,\mu,\nu}(x_1-x_1'\mid \infty)
		\,q^{y_1(x_1'+1)}
		=
		\nu^{y_1}\varphi_{q,q/\nu,\nu}(0\mid y_1)\,
		q^{y_1(x_1+1)}.
	\end{equation*}
\end{lemma}
\begin{proof}
	Expand the definition of $\varphi$
	\eqref{phi_qmunu_definition}, \eqref{phi_qmunu_infinity_definition},
	and rewrite the claim as
	\begin{equation*}
		\mu^{y_1}
		\frac{(q/\mu;q)_{y_1}}{(\nu;q)_{y_1}}
		\sum_{x_1'=-\infty}^{x_1}
		\mu^{x_1-x_1'}
		\frac{(\nu/\mu;q)_{x_1-x_1'}}{(q;q)_{x_1-x_1'}}
		\frac{(\mu;q)_{\infty}}{(\nu;q)_{\infty}}
		\,q^{y_1(x_1'+1)}
		=
		\nu^{y_1}
		\frac{(q/\nu;q)_{y_1}}{(\nu;q)_{y_1}}\,
		q^{y_1(x_1+1)}.
	\end{equation*}
	This simplifies to
	\begin{equation*}
		\sum_{d=0}^{\infty}
		(\mu q^{-y_1})^{d}
		\frac{(\nu/\mu;q)_{d}}{(q;q)_{d}}
		\frac{(\mu q^{-y_1};q)_{\infty}}{(\nu q^{-y_1};q)_{\infty}}
		=1,\qquad d=x_1-x_1',
	\end{equation*}
	which is simply
	$\sum_{d=0}^{\infty}\varphi_{q,\mu q^{-y_1},\nu q^{-y_1}}(d\mid\infty)=1$.
	Note that the series converges for sufficiently small $\mu$.
\end{proof}

The induction step $N-1\to N$ is based on the following lemma:
\begin{lemma}
	\label{lemma:main_identity}
	For all nonnegative integers $g,\ell,y$ there exists
	$\mu_0>0$ such that for
	all $0<\mu<\mu_0$ we have
	\begin{equation}
		\label{main_identity_equation}
		\begin{split}
			&
			\sum_{t=0}^{y}
			\frac{(q; q)_{y}}{(q; q)_{t}(q; q)_{y-t}}\,
			q^{(g+1)t}
			\nu^{y-t} (\nu^{-1}q; q)_{y-t}
			(\nu^{2}q^{-1}; q)_{t}
			\\&
			\hspace{50pt}
			=
			\sum_{s=0}^{y}
			\frac{(q; q)_{y}}{(q; q)_{s}(q; q)_{y-s}}\,
			q^{(g+1-\ell)s}
			\mu^{y-s}
			(\mu^{-1}q; q)_{y-s}
			(\mu \nu q^{-1}; q)_{s}
			\sum_{L=0}^{\infty} q^{-L(y-s)} \mathbf{P}_{\ell, g}(L).
		\end{split}
	\end{equation}
\end{lemma}
We prove \Cref{lemma:main_identity} in the next \Cref{sub:lemma_proof}.

\begin{proof}[Proof of \Cref{thm:push_duality} modulo \Cref{lemma:main_identity}]
	Denote $y=y_N$ and $\bar k=k-y$.
	Write
	for \eqref{push_duality_formulation_equation_detail}:
	\begin{align*}
		&
		\mu^k
		\sum_{\vec x'\in \mathbb{X}^{N}}
		\sum_{\vec y'\in \mathbb{Y}^{N}_k}
		\,\prod_{i=1}^N
		\mathbf{P}_{x_{i-1}-x_{i-1}',x_{i-1}-x_i-1}(x_i-x_i')
		\prod_{i=1}^{N}
		\varphi_{q,q/\mu,\nu}
		\left( \sum\nolimits_{j=0}^{i-1}(y_j'-y_j)\;\Big\vert\;y_i \right)
		\prod_{r=0}^{N}q^{y_r'(x_r'+r)}
		\\&\hspace{2pt}
		=
		\mu^{\bar k}
		\sum_{\vec x'\in \mathbb{X}^{N-1}}
		\sum_{\vec y'\in \mathbb{Y}^{N-1}_{\bar k}}
		\prod_{i=1}^{N-1}\mathbf{P}_{x_{i-1}-x_{i-1}',x_{i-1}-x_i-1}(x_i-x_i')
		\prod_{i=1}^{N-1}\varphi_{q,q/\mu,\nu}\left( \sum\nolimits_{j=0}^{i-1}(y_j'-y_j)\;\Big\vert\;y_i \right)
		\prod_{r=0}^{N-1}q^{y_r'(x_r'+r)}
		\\
		&\hspace{80pt}\times
		\mu^{y}
		\sum_{s=0}^{y}\varphi_{q,q/\mu,\nu}(s\mid y)\, q^{s(x_{N-1}'+N-1)}
		\sum_{L=0}^{\infty}
		q^{(y-s)(x_N-L+N)}\,\mathbf{P}_{x_{N-1}-x_{N-1}',x_{N-1}-x_N-1}(L).
	\end{align*}
	Here we used the notation $s=y_N-y_N'=\sum_{j=0}^{N-1}(y_j'-y_j)$ and $L=x_N-x_N'$. The
	factor $q^{s(x_{N-1}'+N-1)}$ appears on the last line because in
	the sum over $\vec{y}\,'\in\mathbb{Y}^{N-1}_{\bar k}$ in the second line
	the quantity $y_{N-1}'$ is different from $y_{N-1}'$ in the first line:
	the former not take
	into account $s$ Boson particles coming from $y_N$.
	Continuing the computation, we can now apply \Cref{lemma:main_identity}
	with $\ell=x_{N-1}-x_{N-1}'$, $g=x_{N-1}-x_N-1$.
	By the induction hypothesis:
	\begin{align*}
		&=
		\mu^{\bar k}
		\sum_{\vec x'\in \mathbb{X}^{N-1}}
		\sum_{\vec y'\in \mathbb{Y}^{N-1}_{\bar k}}
		\prod_{i=1}^{N-1}\mathbf{P}_{x_{i-1}-x_{i-1}',x_{i-1}-x_i-1}(x_i-x_i')
		\prod_{i=1}^{N-1}\varphi_{q,q/\mu,\nu}\left( \sum\nolimits_{j=0}^{i-1}(y_j'-y_j)\;\Big\vert\;y_i \right)
		\prod_{r=0}^{N-1}q^{y_r'(x_r'+r)}
		\\
		&\hspace{82pt}\times
		q^{y(x_N+N)}
		\nu^y
		\sum_{t=0}^{y}
		q^{(x_{N-1}-x_N-1)t}\,
		\varphi_{q,q/\nu,\nu}(t\mid y)
		\\&
		=
		\nu^{k}\sum_{\vec y''\in \mathbb{Y}^{N-1}_{\bar k}}
		\sum_{t=0}^y
		\varphi_{q,q/\nu,\nu}(t\mid y)
		\prod_{i=1}^{N-1}\varphi_{q,q/\nu,\nu}
		\biggl( \sum\limits_{j=0}^{i-1}(y_j''-y_j)\;\Big\vert\;y_i \biggr)
		q^{t(x_{N-1}+N-1)+(y-t)(x_N+N)}
		\prod_{r=0}^{N-1}q^{y_r''(x_r+r)},
	\end{align*}
	where $t=y_N-y_N''=\sum_{j=0}^{N-1}(y_j''-y_j)$.
	We can rewrite the sum over $\vec y\,''\in \mathbb{Y}^{N-1}_{\bar k}$ and $t$
	as a sum over $\vec y\,''\in \mathbb{Y}^N_k$, thus arriving at the right-hand
	side of \eqref{push_duality_formulation_equation_detail}.
	Throughout the whole computation, all infinite
	series converge for sufficiently small $\mu$
	because all $y_i$, $y_i'$, $y_i''$ are bounded.
	This completes the proof of the duality modulo \Cref{lemma:main_identity} which we establish
	in the next subsection.
\end{proof}

\subsection{Proof of \Cref{lemma:main_identity}}
\label{sub:lemma_proof}

First, note that the infinite sum
over $L$ in the right-hand side of \eqref{main_identity_equation}
is the only part of this identity which can bring convergence issues.
However, because the $\psi$ term in \eqref{P_ell_g_update_probability_definition}
contains $\mu^{L}$, the sum over $L$ in
\eqref{main_identity_equation} indeed converges for sufficiently small $\mu$.

\medskip\noindent\textbf{Step 1} (A rational identity).
We start with the right-hand side of \eqref{main_identity_equation},
and rewrite the sum over $L$ as
\begin{equation}
	\label{main_identity_equation_proof1}
	\begin{split}
		&
		\sum_{L=0}^{\infty}q^{-L(y-s)}\mathbf{P}_{\ell,g}(L)=
		\sum_{L=0}^{\infty}
		\sum_{p=0}^{\min\{\ell, L\}}
		q^{-L(y-s)}
		\varphi_{q^{-1}, q^{g}, \mu \nu q^{g-1}}(p \mid \ell)
		\,
		\psi_{q, \nu \mu^{-1} q^{p}, \nu q^{g}, \nu^{2} q^{g+p}}(L-p)
		\\&
		\hspace{130pt}=
		\sum_{p=0}^{\ell}
		\sum_{m=0}^{\infty}
		q^{-(m+p)(y-s)}
		\varphi_{q^{-1}, q^{g}, \mu \nu q^{g-1}}(p \mid \ell)
		\,
		\psi_{q, \nu \mu^{-1} q^{p}, \nu q^{g}, \nu^{2} q^{g+p}}(m).
	\end{split}
\end{equation}
We use one of Heine's transformation formulas
\cite[(III.2)]{GasperRahman},
\begin{equation*}
	_2\phi_1
	\left[
		\begin{array}{cc} a & b \\ \multicolumn{2}{c}{c} \end{array};
		q, z
	\right]
	=
	\frac{(c/b;q)_{\infty}(bz;q)_{\infty}}{(c;q)_{\infty}(z;q)_{\infty}}\,
	{}_2\phi_1
	\left[
		\begin{array}{cc} abz/c & b \\ \multicolumn{2}{c}{bz} \end{array};
		q, c/b
	\right],
\end{equation*}
with $a=\nu q^g$, $b=\nu\mu^{-1}q^p$, $c=\nu^2q^{g+p}$, and $z=\mu q^{s-y}$,
to rewrite
\begin{equation}\label{main_identity_equation_proof2}
	\mathrm{RHS}\,
	\eqref{main_identity_equation_proof1}
	=
	\sum_{p=0}^{\ell}
	q^{p(s-y)}
	\varphi_{q^{-1}, q^{g}, \mu \nu q^{g-1}}(p \mid \ell)
	\,
	\frac{(\nu q^{p+s-y};q)_{y-s}}
	{(\mu q^{s-y};q)_{y-s}}
	\sum_{r=0}^{y-s}
	\bigl(\mu\nu q^{g}\bigr)^r\,
	\frac{(q^{s-y};q)_r (\nu\mu^{-1}q^p;q)_r}{(\nu q^{p+s-y};q)_r(q;q)_r}.
\end{equation}
The advantage is that now the $q$-hypergeometric sum over $r$ terminates.
Thus, to prove the desired identity
\eqref{main_identity_equation} we need to establish the following,
\begin{equation}
	\label{main_identity_equation_proof3}
	\begin{split}
		&
		\sum_{t=0}^{y}
		\frac{(q; q)_{y}}{(q; q)_{t}(q; q)_{y-t}}\,
		q^{(g+1)t}
		\nu^{y-t} (\nu^{-1}q; q)_{y-t}
		(\nu^{2}q^{-1}; q)_{t}
		\\&\hspace{5pt}=
		\sum_{s=0}^{y}
		\sum_{r=0}^{y-s}
		\sum_{p=0}^{\ell}
		\frac{(q; q)_{y}}{(q; q)_{s}(q; q)_{y-s}}\,
		q^{(g+1-\ell)s}
		(\mu \nu q^{g-p})^{y-s}
		(\nu^{-1} q^{-g})^{y-s-r}
		\\
		&\hspace{80pt}
		\times
		(\mu^{-1}q; q)_{y-s}
		(\mu \nu q^{-1}; q)_{s}\,
		\varphi_{q^{-1}, q^{g}, \mu \nu q^{g-1}}(p \mid \ell)\,
		\varphi_{q,  \nu \mu^{-1} q^{p}, \mu^{-1} q}(y-s-r \mid y-s).
	\end{split}
\end{equation}
where we packed parts of \eqref{main_identity_equation_proof2}
into the second
$\varphi$ expression.
Note that now this is an identity of rational functions not involving infinite summation,
so we do not need to worry about convergence.

\medskip\noindent\textbf{Step 2} (Induction base).
We will prove \eqref{main_identity_equation_proof3}
by induction on $\ell$, but this requires a number of additional transformations.
The base of the induction $\ell=0$ is
\begin{multline*}
	\sum_{t=0}^{y}
	\frac{(q; q)_{y}}{(q; q)_{t}(q; q)_{y-t}}\,
	q^{(g+1)t}
	\nu^{y-t} (\nu^{-1}q; q)_{y-t}
	(\nu^{2}q^{-1}; q)_{t}
	\\
	=
	\sum_{t=0}^{y}
	\frac{(q; q)_{y}}{(q; q)_{t}(q; q)_{y-t}}\,
	q^{(g+1)t} (\nu^{-1}q; q)_{y-t}
	\sum_{r=0}^{t} \nu^{y-t} (\nu^{2}q^{-1}; q)_{t} \,
	\varphi_{q, \mu \nu q^{-1}, \nu^{2}q^{-1}}(r\mid t),
\end{multline*}
where we have set $t=r+s$ in the right-hand side.
This holds because the $\varphi$'s sum to one, cf.
\eqref{phi_qmunu_sum}.

\medskip\noindent\textbf{Step 3} (Setting $\nu=q^{-x}$).
We now turn to the induction step in the proof of \eqref{main_identity_equation_proof3}.
For any $\ell\ge0$,
both sides of this identity are rational functions in $\nu$,
so it suffices to prove the identity for infinitely many values of $\nu$.
We will show it for $\nu = q^{-x}$ for large enough positive integers $x$.
Use the ``self-duality'' property of $\varphi$
\eqref{phi_qmunu_interchange} to write
\begin{equation}
	\label{main_identity_equation_proof4_RHS_rewritten}
	\begin{split}
		\mathrm{RHS}\,\eqref{main_identity_equation_proof3}
		&=
		\sum_{s=0}^{y}
		\sum_{p=0}^{\ell}
		\frac{(q; q)_{y}}{(q; q)_{s}(q; q)_{y-s}}\,
		q^{(g+1-\ell)s}
		(\mu q^{g-p-x})^{y-s}
		(\mu q^{-x-1}; q)_{s}\,
		\varphi_{q^{-1}, q^{g}, \mu q^{-x+g-1}}(p \mid \ell)
		\\
		&\hspace{80pt}
		\times
		(\mu^{-1}q; q)_{y-s}
		\sum_{d=0}^{x-g}
		q^{(y-s)d}
		\varphi_{q,  \mu^{-1} q^{p-x}, \mu^{-1} q}(d \mid x-g).
	\end{split}
\end{equation}

\medskip\noindent\textbf{Step 4} (A $q$-exponential generating series).
Denote
\begin{equation*}
	\Pi(\alpha,\beta):={}_1\phi_0(\beta;q,\alpha)=
	\sum_{y=0}^{+\infty}\frac{(\beta;q)_y}{(q;q)_y}\alpha^y
	=
	\frac{(\alpha \beta;q)_{\infty}}{(\alpha;q)_\infty}
\end{equation*}
(the last equality is the $q$-binomial theorem).
Note that
\begin{equation}
	\label{Pi_properties}
	\Pi(\alpha,\beta)^{-1}=\Pi(\alpha\beta,1/\beta),
	\qquad
	\Pi(\alpha/q,\beta)-\Pi(\alpha,\beta)=
	(1-\beta)\frac{\alpha}{q}\,\Pi(\alpha/q,\beta q).
\end{equation}

Let $\chi$ be a formal parameter.
Multiply both sides of the desired identity
\begin{equation*}
	\mathrm{LHS}\,\eqref{main_identity_equation_proof3}=\mathrm{RHS}\,\eqref{main_identity_equation_proof4_RHS_rewritten}
\end{equation*}
by $\chi^y/(q;q)_y$ and sum over $y$ from $0$ to $+\infty$.
We obtain the following identity that we need to establish:
\begin{equation}\label{main_identity_equation_proof5}
	\begin{split}
		&
		\Pi(\chi q^{g+1},q^{-2x-1})
		\Pi(\chi q^{-x},q^{x+1})
		\Pi(\chi \mu q^{g-x-\ell},\mu^{-1} q^{x+1})
		\\
		&\hspace{20pt}
		=
		\sum_{z=0}^{\infty}
		\frac{\chi^{z}(\mu^{-1}q; q)_{z}}{(q; q)_{z}}
		\sum_{p=0}^{\ell}
		(\mu q^{g-p-x})^{z}\,
		\varphi_{q^{-1}, q^{g}, \mu q^{-x+g-1}}(p \mid \ell)
		\sum_{d=0}^{x-g}
		q^{z d}\,
		\varphi_{q,  \mu^{-1} q^{p-x}, \mu^{-1} q}(d \mid x-g)
		.
	\end{split}
\end{equation}
The third factor $\Pi$ in the left-hand side
of
\eqref{main_identity_equation_proof5}
arises by applying
the first identity in \eqref{Pi_properties}
to $\Pi(\chi q^{g+1-\ell},\mu q^{-x-1})$
coming from
\eqref{main_identity_equation_proof4_RHS_rewritten}.

\medskip\noindent\textbf{Step 5} (Recursion in $\ell$ for the left-hand side of \eqref{main_identity_equation_proof5}).
Denote the
left-hand side of \eqref{main_identity_equation_proof5}
by $L(\mu,\ell)$. The second identity in \eqref{Pi_properties} implies that
\begin{equation}\label{main_identity_equation_proof6_recursion_relation_LHS}
	L(\mu,\ell+1)-L(\mu,\ell)=
	(1-\mu^{-1}q^{x+1})\chi\mu q^{g-\ell-x-1}
	L(\mu/q,\ell).
\end{equation}
To prove the inductive step
$\ell\to\ell+1$
it now suffices to verify the same recursion relation
for the right-hand side of
\eqref{main_identity_equation_proof5}.

\medskip\noindent\textbf{Step 6} (Recursion in $\ell$ for the right-hand side of \eqref{main_identity_equation_proof5}).
We now aim to
check that the right-hand side of \eqref{main_identity_equation_proof5}
satisfies recursion \eqref{main_identity_equation_proof6_recursion_relation_LHS}.
We will perform this check for each coefficient by $\chi^z$
separately. That is, we need to show that for every fixed
$z\in \mathbb{Z}_{\ge0}$,
\begin{align}
	\nonumber
		&
		\sum_{p=0}^{\ell+1}
		q^{-pz}
		\,
		\varphi_{q^{-1}, q^{g}, \mu q^{-x+g-1}}(p \mid \ell+1)
		\sum_{d=0}^{x-g}
		q^{z d}\,
		\varphi_{q,  \mu^{-1} q^{p-x}, \mu^{-1} q}(d \mid x-g)
		\\
		&\nonumber
		\hspace{60pt}-
		\sum_{p=0}^{\ell}
		q^{-pz}
		\,
		\varphi_{q^{-1}, q^{g}, \mu q^{-x+g-1}}(p \mid \ell)
		\sum_{d=0}^{x-g}
		q^{z d}\,
		\varphi_{q,  \mu^{-1} q^{p-x}, \mu^{-1} q}(d \mid x-g)
		\\&
		\label{main_identity_equation_proof7}
		=
		(1-\mu^{-1}q^{x+1})q^{-\ell}\,
		\frac{q^{-z}-1}{1-\mu^{-1}q}
		\\&\nonumber
		\hspace{80pt}
		\times
		\sum_{p=0}^{\ell}
		q^{-p(z-1)}\,\varphi_{q^{-1},q^g,\mu q^{-x+g-2}}(p\mid \ell)
		\sum_{d=0}^{x-g}
		q^{(z-1)d}\,\varphi_{q,\mu^{-1}q^{p-x+1},\mu^{-1}q^2}(d\mid x-g)
		.
\end{align}
Note that the factor $\chi$ in the right-hand side of \eqref{main_identity_equation_proof6_recursion_relation_LHS}
leads to a shift $z\mapsto z-1$ which combined
with the $q$-shifting of $\mu$
brings certain extra terms
into the right-hand side of the above identity.

\medskip\noindent\textbf{Step 7} (Comparing coefficients by $q^{kz}$).
It now suffices to show that the coefficients by $q^{kz}$ for all
$k\in \mathbb{Z}$ in both sides of \eqref{main_identity_equation_proof7}
are the same.
Let us also set $\gamma=q^{-g}$, this will later serve as
a generic parameter.
This leads to the following identity to be checked:
\begin{equation}
	\label{main_identity_equation_proof8}
	\begin{split}
		&
		\sum_{p=0}^{\ell+1}
		\Bigl(
			\varphi_{q^{-1}, \gamma^{-1}, \mu \gamma^{-1}q^{-x-1}}(p \mid \ell+1)
			\,
			\varphi_{q,  \mu^{-1} q^{p-x}, \mu^{-1} q}(p+k \mid x-g)
			\\&\hspace{140pt}
			-
			\varphi_{q^{-1}, \gamma^{-1}, \mu \gamma^{-1} q^{-x-1}}(p \mid \ell)
			\,
			\varphi_{q,  \mu^{-1} q^{p-x}, \mu^{-1} q}(p+k \mid x-g)
		\Bigr)
		\\&
		=
		\frac{(1-\mu^{-1}q^{x+1})q^{-\ell-k}}{1-q/\mu}
		\sum_{p=0}^{\ell+1}
		\Bigl(
			q^{-1}
			\varphi_{q^{-1},\gamma^{-1},\mu\gamma^{-1} q^{-x-2}}(p\mid \ell)
			\,
			\varphi_{q,\mu^{-1}q^{p-x+1},\mu^{-1}q^2}(p+k+1\mid x-g)
			\\&\hspace{140pt}
			-
			\varphi_{q^{-1},\gamma^{-1},\mu\gamma^{-1} q^{-x-2}}(p\mid \ell)
			\,
			\varphi_{q,\mu^{-1}q^{p-x+1},\mu^{-1}q^2}(p+k\mid x-g)
		\Bigr).
	\end{split}
\end{equation}
Simplifying this identity
and
rewriting it in
a
$q$-hypergeometric notation (cf. \eqref{q_hyp_defn}) for $k\ge0$
(the case $k\le 0$ is considered in a similar manner),
we obtain
\begin{equation}
	\begin{split}
		\label{main_identity_equation_proof9}
	&
	\frac{q^{k+1} \left(1-\gamma  q^l\right) \left(1-\mu  q^x\right)}{\left(\mu-q^{x+1} \right) \left(\gamma-\mu  q^k \right)}
	\,
	{_{4}}\phi_{3} \left[
	\begin{array}{cccc}
		q^{\ell+1} & \mu q^{-x-1} & q^x & \gamma q^{x-k}
		\\
		&\mu q^x  & \gamma q^\ell & q^{-k-1}
	\end{array}
	; q^{-1}, q^{-1} \right]
	\\
	&\hspace{40pt}
	-
	\frac{q^{k-x} \left(1-\mu  q^x\right) \left(\mu-\gamma  q^{l+x+1} \right)}{\left(\mu-q^{x+1} \right) \left(\gamma-\mu  q^k \right)}
	\,
	{_{4}}\phi_{3} \left[
	\begin{array}{cccc}
		q^{\ell} & \mu q^{-x-1} & q^x & \gamma q^{x-k}
		\\
		&\mu q^x  & \gamma q^{\ell-1} & q^{-k-1}
	\end{array}
	; q^{-1}, q^{-1} \right]
	\\&
	=
	\frac{q^{-x} \left(1-q^{k+x+1}\right) \left(\gamma  q^x-q^k\right)}{\left(1-q^{k+1}\right) \left(\gamma -\mu  q^k\right)}
	\,
		{_{4}}\phi_{3} \left[
		\begin{array}{cccc}
			q^{\ell} & \mu q^{-x-2} & q^x & \gamma q^{x-k-1}
			\\
			&\mu q^{x-1}  & \gamma q^{\ell-1} & q^{-k-2}
		\end{array}
		; q^{-1}, q^{-1} \right]
		\\
		&\hspace{90pt}
		-
		{_{4}}\phi_{3} \left[
		\begin{array}{cccc}
			q^{\ell} & \mu q^{-x-2} & q^x & \gamma q^{x-k}
			\\
			&\mu q^{x-1}  & \gamma q^{\ell-1} & q^{-k-1}
		\end{array}
	; q^{-1}, q^{-1} \right].
	\end{split}
\end{equation}
In passing from \eqref{main_identity_equation_proof8} to
\eqref{main_identity_equation_proof9}
we have assumed that $\gamma$ is not an integer power of $q$:
as both sides of \eqref{main_identity_equation_proof8}
are rational in $\gamma$, it suffices to establish \eqref{main_identity_equation_proof8}
for infinitely many values of $\gamma$.
Note that all the $q$-hypergeometric series in \eqref{main_identity_equation_proof9} are terminating.

\medskip\noindent\textbf{Step 8} (Extension and proof of the $q$-hypergeometric identity).
To establish \eqref{main_identity_equation_proof9},
consider its extension for incomplete $q$-hypergeometric series:
\begin{equation*}
	{_{4}}\phi_{3}^{[p]} \left[
		\begin{array}{cccc}
			a & b & c & d
			\\
			&e & f & g
		\end{array}
	; q, z \right]:=
	\sum_{n=0}^{p}
	\frac{(a;q)_p(b;q)_p(c;q)_p(d;q)_p}{(q;q)_p(e;q)_p(f;q)_p(g;q)_p}z^n.
\end{equation*}
Then the right-hand side of the analogue
of \eqref{main_identity_equation_proof9} is nonzero but can be explicitly
computed:
\begin{align}
		&\nonumber
		\frac{q^{k+1} \left(1-\gamma  q^l\right) \left(1-\mu  q^x\right)}{\left(\mu-q^{x+1} \right) \left(\gamma-\mu  q^k \right)}
		\,
		{_{4}}\phi_{3}^{[p]} \left[
		\begin{array}{cccc}
			q^{\ell+1} & \mu q^{-x-1} & q^x & \gamma q^{x-k}
			\\
			&\mu q^x  & \gamma q^\ell & q^{-k-1}
		\end{array}
		; q^{-1}, q^{-1} \right]
		\\ \nonumber
		&\hspace{40pt}
		-
		\frac{q^{k-x} \left(1-\mu  q^x\right) \left(\mu-\gamma  q^{l+x+1} \right)}{\left(\mu-q^{x+1} \right) \left(\gamma-\mu  q^k \right)}
		\,
		{_{4}}\phi_{3}^{[p]} \left[
		\begin{array}{cccc}
			q^{\ell} & \mu q^{-x-1} & q^x & \gamma q^{x-k}
			\\
			&\mu q^x  & \gamma q^{\ell-1} & q^{-k-1}
		\end{array}
		; q^{-1}, q^{-1} \right]
		\\&
		\label{main_identity_equation_proof10}
		\hspace{40pt}
		-
		\frac{q^{-x} \left(1-q^{k+x+1}\right) \left(\gamma  q^x-q^k\right)}{\left(1-q^{k+1}\right) \left(\gamma -\mu  q^k\right)}
		\,
		{_{4}}\phi_{3}^{[p]} \left[
		\begin{array}{cccc}
			q^{\ell} & \mu q^{-x-2} & q^x & \gamma q^{x-k-1}
			\\
			&\mu q^{x-1}  & \gamma q^{\ell-1} & q^{-k-2}
		\end{array}
		; q^{-1}, q^{-1} \right]
		\\ \nonumber
		&\hspace{40pt}
		+
		{_{4}}\phi_{3}^{[p]} \left[
		\begin{array}{cccc}
			q^{\ell} & \mu q^{-x-2} & q^x & \gamma q^{x-k}
			\\
			&\mu q^{x-1}  & \gamma q^{\ell-1} & q^{-k-1}
		\end{array}
	; q^{-1}, q^{-1} \right]
	\\& \nonumber
	\hspace{80pt}
		=\frac{q^{k+1} (q^\ell;q^{-1})_p (q^x;q^{-1})_{p+1} (q^{-x-1} \mu ;q^{-1})_{p+1}
		(q^{x-k} \gamma ; q^{-1})_{p+1}}
		{(\gamma-\mu  q^k)(\mu-q^{x+1})(q^{-1};q^{-1})_p
		 (q^{-k-1};q^{-1})_{p+1}
		 (q^{\ell-1} \gamma ;q^{-1})_p (q^{x-1} \mu ;q^{-1})_p}.
\end{align}
This last identity is readily proven by induction
on $p$. Indeed, both
\begin{equation*}
	\textnormal{
	$\mathrm{LHS}\,\eqref{main_identity_equation_proof10}\,[p+1]-\mathrm{LHS}\,\eqref{main_identity_equation_proof10}\,[p]$
	\qquad and\qquad
	$\mathrm{RHS}\,\eqref{main_identity_equation_proof10}\,[p+1]-\mathrm{RHS}\,\eqref{main_identity_equation_proof10}\,[p]$
	}
\end{equation*}
are simple sums of
ratios of $q$-Pochhammer symbols,
and their equality is checked directly.
Taking any $p\ge\ell+1$ makes the right-hand side of
\eqref{main_identity_equation_proof10} vanish,
and leads to \eqref{main_identity_equation_proof9}.

\medskip

This completes the proof of
\Cref{lemma:main_identity}.\qed

\subsection{Contour integral observables}
\label{sub:push_contour_integrals}

In this subsection we utilize the
duality of \Cref{thm:push_duality}
to obtain nested contour integral formulas for the
$q$-moments of the $q$-Hahn PushTASEP.
Fix $k\ge1$, and denote
$\vec n=(n_1\ge \ldots\ge n_k\ge0 )$ with $n_1\le N$.
Consider the joint moment
\begin{equation}
	\label{push_q_moment_definition}
	u(t;\vec n):=
	\mathop{\mathbb{E}}
	\biggl[
		\prod_{j=1}^{k}
		q^{x_{n_j}(t)+n_j}
	\biggr]
	,
	\qquad
	t=0,1,2,\ldots ,
\end{equation}
where the $q$-Hahn PushTASEP starts from the step initial
configuration $x_i(0)=-i$, $i=1,2,\ldots$.
First, let us deal with convergence of the expectation \eqref{push_q_moment_definition}.

\begin{lemma}
	\label{lemma:finiteness_q_moments}
	When $0<\mu<q^k$ and the other $q$-Hahn PushTASEP parameters satisfy
	\eqref{mu_nu_parameters_for_qHahn}, the $q$-moment
	$u(t;\vec n)$ is finite for all $t\in \mathbb{Z}_{\ge0}$.
\end{lemma}
\begin{proof}
	By the definition of the process in \Cref{sub:def_nonnegativity},
	the $q$-Hahn pushTASEP one-step
	transition probability $\vec x\to \vec x'$ can be bounded from above by
	$\mathsf{polynomial}(\vec x-\vec x')\cdot \mu^{\sum_{i=1}^{N}(x_i-x_i')}$.
	Multiplying this estimate by
	$\prod_{j=1}^{k}
	q^{x_{n_j}(t)+n_j}$
	and summing over $\vec x'$ (which can take arbitrarily large negative values)
	we get a finite sum if $\mu<q^k$.
\end{proof}
The bound $\mu<q^k$ in \Cref{lemma:finiteness_q_moments}
cannot be relaxed as the $k$-th $q$-moment of the first particle after the first step
has the form
\begin{equation*}
	\mathop{\mathbb{E}}\bigl[q^{k (x_1(1)+1)}\bigr]
	=
	\sum_{\ell=0}^{\infty}q^{-k\ell}\,\varphi_{q,\mu,\nu}(\ell\mid\infty)
	=
	\frac{(\mu;q)_{\infty}}{(\nu;q)_{\infty}}
	\sum_{\ell=0}^{\infty}
	\left(
		\mu q^{-k}
	\right)^{\ell}
	\frac{(\nu/\mu;q)_{\ell}}{(q;q)_{\ell}}
	=
	\frac{(\nu q^{-k};q)_k}{(\mu q^{-k};q)_{k}}
	=
	\frac{(\nu/q;q^{-1})_k}{(\mu/q;q^{-1})_k},
\end{equation*}
and this series converges
only when $|\mu q^{-k}|<1$.

\begin{remark}
	\Cref{lemma:finiteness_q_moments}
	implies that
	in
	\Cref{thm:push_duality}
	and
	\Cref{lemma:induciton_base,lemma:main_identity} we can take
	$\mu_0=q^k$.
\end{remark}

\begin{theorem}
	\label{thm:push_q_moments_in_the_text}
	When $0<\mu<q^k$ and the other $q$-Hahn PushTASEP parameters satisfy
	\eqref{mu_nu_parameters_for_qHahn},
	we have
	\begin{multline}\label{push_q_moment_formula_in_the_text}
		u(t;\vec n)=
		\frac{(-1)^k q^{\frac{k(k-1)}{2}}}{(2\pi\sqrt{-1})^k}
		\oint\frac{dz_1}{z_1}
		\ldots
		\oint\frac{dz_k}{z_k}
		\prod_{1\le A<B\le k}
		\frac{z_A-z_B}{z_A-qz_B}
		\\\times
		\prod_{j=1}^{k}
		\Biggl[
		\left(
			\frac{1-\nu z_j}{1-z_j}
		\right)^{n_j}
		\left(
			\frac{1-\nu q^{-1}z_j^{-1}}
			{1-\mu q^{-1}z_j^{-1}}
		\right)^t
		\frac{1}{1-\nu z_j}
		\Biggr]
	\end{multline}
	for all $\vec n=(n_1\ge \ldots\ge n_k\ge0 )$.
	Here each contour for $z_A$ is a
	simple closed curve around $1$ which
	encircles the contour $qz_B$ for $B>A$, but not the points
	$\mu/q$ or $1/\nu$.
	See \Cref{fig:push_contours} for an illustration.
\end{theorem}
The condition $\mu<q^{-k}$ also implies the existence of the $k$
nested integration contours in \eqref{push_q_moment_formula_in_the_text}.

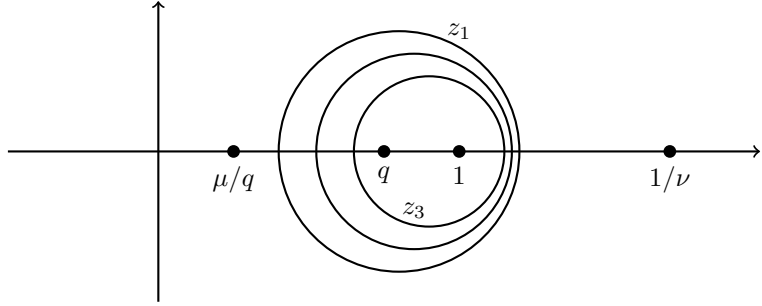
\begin{figure}[htpb]
	\centering
	\begin{tikzpicture}
		[scale=2, thick]
		\draw[->] (0,-1)--++(0,2);
		\draw[->] (-1,0)--++(5,0);
		\draw[fill] (2,0) circle (1pt) node[below,yshift=-2] {1};
		\draw[fill] (1.5,0) circle (1pt) node[below,yshift=-2] {$q$};
		\draw[fill] (3.4,0) circle (1pt) node[below,yshift=-2] {$1/\nu$};
		\draw[fill] (.5,0) circle (1pt) node[below,yshift=-2] {$\mu/q$};
		\draw (1.6,0) circle (.8);
		\node at (2,.8) {$z_1$};
		\draw (1.7,0) circle (.65);
		\draw (1.8,0) circle (.5);
		\node at (1.7,-.39) {$z_3$};
	\end{tikzpicture}
	\caption{Possible integration contours for the $q$-moments of the
	$q$-Hahn PushTASEP \eqref{push_q_moment_formula_in_the_text}
	with $k=3$.}
	\label{fig:push_contours}
\end{figure}

\begin{proof}[Proof of \Cref{thm:push_q_moments_in_the_text}]
	We establish this theorem by showing that
	both sides of \eqref{push_q_moment_formula_in_the_text} satisfy
	certain free evolution equations with two-body boundary conditions.
	This approach to obtaining $q$-moment formulas
	was applied for $q$-TASEPs and ASEP in
	\cite{BorodinCorwinSasamoto2012},
	\cite{BorodinCorwin2013discrete},
	and for the $q$-Hahn TASEP (\Cref{thm:TASEP_moments})
	in \cite{Corwin2014qmunu}.

	Start with the right-hand side of \eqref{push_q_moment_formula_in_the_text}
	and denote it
	by $v(t;\vec n)$, where $\vec n=(n_1,\ldots,n_k )$, $n_i\ge0$, are not necessarily weakly decreasing.
	We need to show that $u(t;\vec n)=v(t;\vec n)$ for weakly decreasing $n_1\ge \ldots\ge n_k\ge0$.
	Let
	\begin{equation*}
		\nabla^j_{a,b}f(\vec n):=a f(n_1,\ldots,n_k)+b f(n_1,\ldots,n_{j-1},n_j-1 ,n_{j+1},\ldots,n_k ).
	\end{equation*}
	Similarly to
	\cite{BorodinCorwin2013discrete},
	\cite{Corwin2014qmunu}
	one can readily check that the contour integrals $v(t;\vec n)$ satisfy the
	\emph{free evolution equations}
	\begin{equation}\label{free_ev_equations}
		\prod_{j=1}^k
		\nabla^{j}_{\mu-q,q-\mu\nu}v(t+1; n_1,\ldots,n_k )=
		\prod_{j=1}^{k}
		\nabla^{j}_{\nu-q,q-\nu^2}v(t;n_1,\ldots,n_k )
	\end{equation}
	with the boundary conditions
	\begin{enumerate}[\bf1.\/]
		\item $v(t;\vec n)=0$ if $n_k=0$;
		\item $v(0;\vec n)=1$ if $n_1\ge \ldots\ge n_k>0 $;
		\item If $n_i=n_{i+1}$ for some $i=1,\ldots,k-1$, then
			\begin{equation}
				\label{two_body_boundary}
				\begin{split}
				&
				\frac{\nu(1-q)}{1-q\nu}v(t;n_1,\ldots,n_i-1,n_{i+1}-1,\ldots,n_k  )
				+
				\frac{q-\nu}{1-q \nu}v(t;n_1,\ldots,n_i,n_{i+1}-1 ,\ldots,n_k )
				\\
				&\hspace{60pt}
				+
				\frac{1-q}{1-q\nu}v(t;n_1,\ldots,n_i,n_{i+1},\ldots ,n_k )
				-
				v(t;n_1,\ldots,n_i-1,n_{i+1},\ldots,n_k)=0.
				\end{split}
			\end{equation}
	\end{enumerate}
	In more detail, the equations \eqref{free_ev_equations} are
	satisfied by the integrand in \eqref{push_q_moment_formula_in_the_text},
	and the boundary conditions
	require contour integration.
	In particular, combining the integrals
	as in
	\eqref{two_body_boundary}
	gives rise to a factor $z_i-qz_{i+1}$ under the integral which
	cancels the corresponding factor in the double product over $A<B$.
	The integrand then becomes skew symmetric in $z_i$ and $z_{i+1}$,
	while the $z_i$ and $z_{i+1}$ integration contours can be chosen to coincide.
	This implies that the combination \eqref{two_body_boundary}
	of the contour integrals vanishes.

	Next, from \cite{Povolotsky2013} or \cite{Corwin2014qmunu}
	(up to a notation change) it follows
	that
	for any function $v(t;\vec n)$ satisfying the two-body boundary conditions
	\eqref{two_body_boundary} we have
	\begin{equation*}
		\prod_{j=1}^k \nabla^{j}_{1-p,p} v(t;\vec n)=
		\mathsf{P}^{\textnormal{Boson}}_{q,(1-\nu)p+\nu,\nu}v(t;\vec n).
	\end{equation*}
	Therefore, the free evolution equations \eqref{free_ev_equations}
	together with the two-body boundary conditions
	\eqref{two_body_boundary}
	are equivalent to the \emph{true evolution equations}
	\begin{equation}
		\label{true_ev_eq}
		\mu^k
		\mathsf{P}^{\textnormal{Boson}}_{q,q/\mu,\nu}v(t+1;\vec n)
		=
		\nu^k
		\mathsf{P}^{\textnormal{Boson}}_{q,q/\nu,\nu}v(t;\vec n).
	\end{equation}

	Finally, from the duality (\Cref{thm:push_duality})
	it follows that the $q$-moments
	$u(t;\vec n)$ satisfy the same true evolution equations
	\eqref{true_ev_eq} (the time evolution $t\to t+1$
	corresponds to the application of the one-step
	transition operator
	$\mathsf{P}^{\textnormal{PushTASEP}}_{q,\mu,\nu}$).
	Moreover, $u(t;\vec n)$ clearly satisfy the remaining boundary conditions \textbf{1}~and~\textbf{2} above
	(recall that, by agreement, $x_0\equiv +\infty$).
	The uniqueness of the solution to the
	true evolution equations with the boundary conditions \textbf{1}~and~\textbf{2}
	follows from the invertibility of the $q$-Boson operator
	based on its spectral theory \cite{BCPS2014},
	\cite{CorwinPetrov2015}.
	Hence $u(t;\vec n)=v(t;\vec n)$ for all $N\ge n_1\ge \ldots\ge n_k\ge0 $,
	as desired.
\end{proof}

Although only finitely many of the $q$-moments of the
$q$-Hahn PushTASEP are finite,
based on them
we conjecture
a Fredholm determinantal
formula\footnote{We will not
recall the definition of a Fredholm determinant of a kernel on a contour,
see, e.g., \cite{Bornemann_Fredholm2010} or
one of the books
\cite{Lax2002book},
\cite{Simon-trace-ideals},
\cite{GohbergKrein1969}.}
for the $e_q$-Laplace transform
of the single particle location in the process.
When $\nu=0$, the Fredholm determinant identity is proven
rigorously
\cite{BorodinCorwinFerrariVeto2013}
using the formalism of $q$-Whittaker measures and symmetric functions
instead of duality and moment formulas.
A~duality-based proof for the continuous time $q$-PushTASEP (i.e., $\nu=0$ and $\mu\to0$ in our notation)
is also possible, cf.
\cite[Theorem 7.10]{MatveevPetrov2014}.

\begin{conjecture}\label{conj:conjecture_push_Fredholm}
	For the $q$-Hahn PushTASEP started from the step initial configuration
	we have
	\begin{equation*}
		\mathop{\mathbb{E}}
		\biggl[
			\frac{1}{\left( \zeta q^{x_n(t)+n};q \right)_{\infty}}
		\biggr]
		=
		\det\left( I+K_{\zeta} \right),
		\qquad \zeta\in \mathbb{C}\setminus \mathbb{R}_{>0}.
	\end{equation*}
	Here $K_{\zeta}$ is a kernel
	of an integral operator
	on
	a small positively oriented circle around $1$
	having the form
	\begin{equation*}
		K_\zeta(w,w')=\frac{1}{2\pi \sqrt{-1}}
		\int_{-\infty\sqrt{-1}+\frac12}^{\infty\sqrt{-1}+\frac12}
		\frac{\pi}{\sin(-\pi s)}
		\hspace{.3pt}
		(-\zeta)^s
		\frac{g(w)}{g(q^s w)}
		\frac{ds}{q^s w-w'},
	\end{equation*}
	with
	\begin{equation*}
		g(w)=
		\left(
			\frac{(\nu w;q)_{\infty}}{(w;q)_{\infty}}
		\right)^n
		\left(
			\frac{(\mu w^{-1};q)_{\infty}}{(\nu w^{-1};q)_{\infty}}
		\right)^t
		\frac{1}{(\nu w;q)_\infty}
		.
	\end{equation*}
\end{conjecture}
A direct proof of this formula
by expanding $1/(\zeta q^{x_t(t)+n};q)_\infty$
as a series in $\zeta$ close to $0$,
and interchanging the summation and the expectation
is not possible as
the random variable
$x_n(t)+n$ has only finitely many moments.
(However, direct proofs work for related processes
like $q$-TASEP and ASEP, cf. \cite{BorodinCorwin2011Macdonald}, \cite{BorodinCorwinSasamoto2012}.)
It would be very interesting to find an extension of the symmetric functions
formalism used in \cite{BorodinCorwinFerrariVeto2013} in order to establish
\Cref{conj:conjecture_push_Fredholm}.
Another way around this obstacle which 
leads to observables suitable for asymptotic analysis 
was suggested in
\cite{imamura2017fluctuations},
\cite{imamura2019stationary}.
Overall, we believe that our conjecture can be established with the help of
a good notion of analytic continuation from known Fredholm determinantal formulas.

If \Cref{conj:conjecture_push_Fredholm} or another family of suitable observables
is available, then we expect the $q$-Hahn PushTASEP to display
the common behavior characteristic of the KPZ universality class.
That is, as $t\to+\infty$, the height function divided by $t$ 
should
have a limit shape, and the (single-point)
fluctuations of the height function
around this limit shape should have scale $t^{1/3}$ and be governed by
the GUE Tracy--Widom distribution.
The corresponding results for the $q$-TASEP can be found in \cite{FerrariVeto2013}, 
\cite{barraquand2015phase}.

\section{Beta limit}
\label{sec:beta_limit}

In this section we consider the limit of our $q$-Hahn PushTASEP
as $q,\mu,\nu\to1$. A similar limit of the $q$-Hahn TASEP
was discovered in \cite{CorwinBarraquand2015Beta}.
The latter is related to the distribution of the
random walk in beta-distributed random environment.
From the $q$-Hahn PushTASEP we obtain a more complicated model of polymer type.
It is unclear whether this new model is related to a random walk in random environment.

\subsection{Definition of the limiting model}
\label{sub:def_beta_limit_model}

Consider the random variables $X(i,t):=q^{-(x_i(t)+i)}$,
where $\vec x(t)$ is the $q$-Hahn PushTASEP with the step initial
condition $x_i(0)=-i$, $i\ge1$. Fix $N$ and view $X(i,t)$ as a
random process with values in $(0,1]$,
indexed by $(i,t)\in \left\{ 1,\ldots,N  \right\}\times \mathbb{Z}_{\ge0}$.
Scale the parameters as
\begin{equation}\label{beta_parameters_scaling}
	\textnormal{$q=e^{-\eps}$, $\mu=e^{-\bar \mu \eps}$,
	$\nu=e^{-\bar\nu \eps}$, where $0<\bar\mu<\bar\nu$, $\bar{\nu}\ge\tfrac{1}{2}$, and $\eps\searrow0$.}
\end{equation}
Note that these scaled $(q,\mu,\nu)$ fall under
the $q$-Hahn PushTASEP parameter restrictions
\eqref{mu_nu_parameters_for_qHahn}.
We will show that as $\eps\to0$, the process $X(i,t)$ converges to a certain
process $Z(i,t)$ defined as follows
using the probability distributions from \Cref{sub:2F1_distributions}.
\begin{definition}
	\label{def:Z_process}
	Fix $\bar{\mu}$ $\bar{\nu}$. Let the random process $Z(i, t)$,
	$(i, t) \in \mathbb{Z}_{> 0} \times \mathbb{Z}_{\geq 0}$, be defined
	recursively by:
	\begin{enumerate}[\bf1.\/]
	\item
	$Z(i, 0) = 1$ for all $i$.
	\item
		Set $Z(1,t) = Z(1,t-1) \cdot \mathscr{B}_1(0, \bar{\mu}, \bar{\nu}-\bar{\mu})$,
		where $\mathscr{B}_1$ is the generalized beta distribution \eqref{gen_beta_def}.
	\item
	For $i >1$ and $t > 0$ with probability one $Z(i, t-1) \neq Z(i-1,t)$.
	Then define
	\begin{equation}
		\label{NBB_recurrence}
		Z(i, t) :=
		\begin{cases}
			Z(i, t-1)\cdot \mathscr{NBB}_1
			\left(
				2\bar{\nu}-1, \frac{Z(i-1, t)^{-1} - Z(i-1,t-1)^{-1}}{Z(i,t-1)^{-1} - Z(i-1,t-1)^{-1}}, \frac{Z(i,t-1)}{Z(i-1, t)}, \bar{\mu}, \bar{\nu}-\bar{\mu}
			\right),\\
			\hspace{230pt}
			\textnormal{if $Z(i, t-1) < Z(i-1,t)$};
			\\[6pt]
			Z(i-1, t)\cdot \mathscr{NBB}_1
			\left(
				2\bar{\nu}-1, \frac{Z(i, t-1)^{-1} - Z(i-1,t-1)^{-1}}{Z(i-1, t)^{-1} - Z(i-1, t-1)^{-1}}, \frac{Z(i-1, t)}{Z(i, t-1)}, \bar{\mu}, \bar{\nu}
			\right),\\
			\hspace{230pt}
			\textnormal{if $Z(i-1,t) < Z(i, t-1)$},
		\end{cases}
	\end{equation}
			where
			$\mathscr{NBB}_1$ is the distribution given by \eqref{nbb1_distribution_definition}.
	\end{enumerate}
\end{definition}
Let us discuss two points related to the definition of
the process $Z(i,t)$.
First, note that when $\bar{\nu}=\frac{1}{2}$, the recurrence \eqref{NBB_recurrence} simplifies:
\begin{equation*}
	Z(i, t) :=
	\begin{cases}
		Z(i, t-1)\cdot \mathscr{B}_1\left(\frac{Z(i,t-1)}{Z(i-1, t)}, \bar{\mu}, \frac{1}{2}-\bar{\mu}\right), & \text{if} \ Z(i, t-1) < Z(i-1,t);
		\\[7pt]
		Z(i-1, t)\cdot \mathscr{B}_1\left(\frac{Z(i-1, t)}{Z(i, t-1)}, \bar{\mu}, \frac{1}{2} \right), & \text{if} \ Z(i-1,t) < Z(i, t-1).
	\end{cases}
\end{equation*}
In particular, there is no immediate dependence on $Z(i-1, t-1)$ in the
recurrence formula.
Moreover, if above we have
$Z(i-1, t)=Z$ and $Z(i, t-1) = \delta Z$, then $Z(i, t)\to Z$ as $\delta \to 1$
because $\mathscr{B}_1$ converges to the delta measure at $1$.

Second, the
definition of $Z(i,t)$ when $Z(i,t-1)=Z(i-1,t)$
also makes sense even if $\bar{\nu}>\frac{1}{2}$, as follows from the next lemma:
\begin{lemma}
	\label{lemma:Z_Z_equal_remark_1_15}
	Let $\bar{\nu}>\frac{1}{2}$.
	Assume that
	\begin{equation*}
		Z(i-1, t-1)=X,
		\qquad
		Z(i-1, t)=Z < X,
		\qquad
		Z(i, t-1) = (1-\gamma)Z.
	\end{equation*}
	Then $Z(i, t) \to Z \cdot \mathscr{B}_1(Z/X, \bar{\mu}, 2\bar{\nu}-1)$ as $\gamma \to 0$.
\end{lemma}
\begin{proof}
	We use
	Euler's transformation formula
\begin{align*}
{_{2}F_{1}}(a, b; c; z) = (1-z)^{c-a-b}{_{2}F_{1}}(c-a, c-b; c; z)
\end{align*}
and the Gauss's theorem
\begin{align*}
{_{2}F_{1}}(a, b; c, 1) = \frac{\Gamma(c)\Gamma(c-a-b)}{\Gamma(c-a)\Gamma(c-b)}, \quad c > \max\{0, a, b, a+b\}.
\end{align*}
For $\gamma \to 0+$ the probability density of
$\frac{Z(i, t)}{(1-\gamma)Z}$
(conditioned on $X$ and $Z$)
at $x$ is
\begin{align*}
	&{_{2}F_{1}}
	\left(
		2\bar{\nu}-1, \bar{\nu};
		\bar{\nu}-\bar{\mu};
		\frac{(Z^{-1}-X^{-1})(1-x)}{(1-(1-\gamma)x)(Z^{-1}(1-\gamma)^{-1} - X^{-1})}
	\right)
	\\
	&\hspace{80pt}
	\times
	\frac{x^{\bar{\mu}-1}(1-x)^{\bar{\nu}-\bar{\mu}-1}}
	{(1-(1-\gamma)x)^{\bar{\nu}}}
	\frac{\gamma^{\bar{\mu}+2\bar{\nu}-1}\Gamma(\bar{\nu})}
	{\Gamma(\bar{\mu}) \Gamma(\bar{\nu}-\bar{\mu})}
	\left(
		\frac{Z^{-1}(1-\gamma)^{-1}}
		{Z^{-1}(1-\gamma)^{-1} - X^{-1}}
	\right)^{2\bar{\nu}-1}
	\\&\hspace{20pt}
	=
	{_{2}F_{1}}
	\left(
		-\bar{\mu}-\bar{\nu}+1, -\bar{\mu}; \bar{\nu}-\bar{\mu};
		\frac{(Z^{-1}-X^{-1})(1-x)}
		{(1-(1-\gamma)x)(Z^{-1}(1-\gamma)^{-1} - X^{-1})}
	\right)
	\\
	&\hspace{80pt}
	\times
	\frac{\gamma^{\bar{\mu}+2\bar{\nu}-1}\Gamma(\bar{\nu})}
	{\Gamma(\bar{\mu}) \Gamma(\bar{\nu}-\bar{\mu})}
	\left(
		\frac{Z^{-1}(1-\gamma)^{-1}}
		{Z^{-1}(1-\gamma)^{-1} - X^{-1}}
	\right)^{2\bar{\nu}-1}
	\\
	&\hspace{80pt}
	\times
	\left(
		\frac{Z^{-1}(1-\gamma)^{-1}-xX^{-1}}
		{(1-(1-\gamma)x)(Z^{-1}(1-\gamma)^{-1} - X^{-1})}
	\right)^{-\bar{\mu}+1-2\bar{\nu}}
	\frac{x^{\bar{\mu}-1}(1-x)^{\bar{\nu}-\bar{\mu}-1}}
	{(1-(1-\gamma)x)^{\bar{\nu}}} \gamma^{-\bar{\mu}+1-2\bar{\nu}}
	\\&\hspace{20pt}
	\to
 \frac{x^{\bar{\mu}-1}(1-x)^{2\bar{\nu}-2}}
 {(1-xZ/X)^{\bar{\mu}+2\bar{\nu}-1}}
 \frac{(1 - Z/X)^{\bar{\mu}+2\bar{\nu}-1}
 \Gamma(\bar{\mu}+2\bar{\nu}-1)}
 {\Gamma(\bar{\mu})\Gamma(2\bar{\nu}-1)}.
\end{align*}
For $\gamma \to 0-$ the probability density of $\frac{Z(i, t)}{Z}$
(again, conditioned on $Z$ and $X$)
at $x$ is
\begin{align*}
	&
	{_{2}F_{1}}\left(
		2\bar{\nu}-1, \bar{\nu}+\bar{\mu}; \bar{\nu}; \frac{(Z^{-1}(1-\gamma)^{-1}-X^{-1})(1-x)}{(1-(1-\gamma)^{-1}x)(Z^{-1} - X^{-1})}
	\right)
	\\&\hspace{80pt}
	\times
	\frac{x^{\bar{\mu}-1}(1-x)^{\bar{\nu}-1}}
	{(1-(1-\gamma)^{-1}x)^{\bar{\nu}+\bar{\mu}}}
	\frac{(\frac{-\gamma}{1-\gamma})^{\bar{\mu}+2\bar{\nu}-1}\Gamma(\bar{\nu}+\bar{\mu})}{\Gamma(\bar{\mu}) \Gamma(\bar{\nu})}
	\left(\frac{Z^{-1}}{Z^{-1} - X^{-1}}\right)^{2\bar{\nu}-1}
	\\&\hspace{20pt}
	=
	{_{2}F_{1}}\left(
		1-\bar{\nu}, -\bar{\mu}; \bar{\nu}; \frac{(Z^{-1}(1-\gamma)^{-1}-X^{-1})(1-x)}{(1-(1-\gamma)^{-1}x)(Z^{-1} - X^{-1})}
	\right)
	\frac{(\frac{-\gamma}{1-\gamma})^{\bar{\mu}+2\bar{\nu}-1}\Gamma(\bar{\nu}+\bar{\mu})}{\Gamma(\bar{\mu}) \Gamma(\bar{\nu})}
	\\&\hspace{80pt}
	\times
	\left(\frac{Z^{-1}}{Z^{-1} - X^{-1}}\right)^{2\bar{\nu}-1}
	\left(\frac{Z^{-1}-xX^{-1}}{(1-(1-\gamma)^{-1}x)(Z^{-1}- X^{-1})}\right)^{-\bar{\mu}+1-2\bar{\nu}}
	\\&\hspace{80pt}
	\times
	\frac{x^{\bar{\mu}-1}(1-x)^{\bar{\nu}-1}}{(1-(1-\gamma)x)^{\bar{\nu}+\bar{\mu}}} \left(\frac{-\gamma}{1-\gamma}\right)^{-\bar{\mu}+1-2\bar{\nu}}
	\\&\hspace{20pt}
	\to
	\frac{x^{\bar{\mu}-1}(1-x)^{2\bar{\nu}-2}}{(1-xZ/X)^{\bar{\mu}+2\bar{\nu}-1}}
	\frac{(1 - Z/X)^{\bar{\mu}+2\bar{\nu}-1}\Gamma(\bar{\mu}+2\bar{\nu}-1)}{\Gamma(\bar{\mu})\Gamma(2\bar{\nu}-1)}.
\end{align*}
This completes the proof.
\end{proof}

\subsection{Change of variables and inverse beta recursion}
\label{sec:cov}
Through a change of variables (pointed out to us by Guillaume Barraquand after the first posting of this work), it is possible to simplify the form of the recursion for $Z$ given in \Cref{def:Z_process}. The generalized negative binomial beta distributions  reduce to their standard counterparts. 
In the case $\bar{\nu}=\tfrac{1}{2}$, the resulting recursion is quite similar, though different from the one satisfied by the inverse beta polymer partition function \cite{thieryLD2015integrable}. 
In particular, the choice of parameters for the beta random variable depends on whether $Z(i-1,t)$ or $Z(i,t-1)$ is greater.
It is not clear whether for $Z(i,t)$ there exists a representation as a polymer
partition function.

Define $\tilde{Z}(i,t) := Z(i,t)^{-1}$ where $Z$ is given through \Cref{def:Z_process}. By combining this change of variables with that of \Cref{lem:changeofvar}, we may rewrite the recursion satisfied by $Z$ as follows.

\begin{lemma}\label{lem:Ztilde}
$\tilde{Z}(i,t) := Z(i,t)^{-1}, (i,t)\in \mathbb{Z}_{>0}\times \mathbb{Z}_{\geq 0}$ satisfies the recursion:
\begin{enumerate}
\item $\tilde{Z}(i,0) = 1$ for all $i$.
\item $\tilde{Z}(1,t) = \tilde{Z}(1,t-1) \,\cdot \, \mathscr{Beta}^{-1}(\bar{\mu},\bar{\nu}-\bar{\mu})$ where $\mathscr{Beta}^{-1}$ is the inverse of a beta distributed random variable (see \Cref{sub:2F1_distributions}).
\item For $i>1$ and $t>0$ with probability one $\tilde{Z}(i,t-1)\neq \tilde{Z}(i-1,t)$. Then, when $\tilde{Z}(i,t-1)>\tilde{Z}(i-1,t)$, 
$$
\tilde{Z}(i,t) = \tilde{Y} \tilde{Z}(i,t-1) + (1-\tilde{Y}) \tilde{Z}(i-1,t),
$$
where $\tilde{Y}$ is $\mathscr{NBBeta}^{-1}(2\bar{\nu}-1,\tfrac{\tilde{Z}(i-1,t)-\tilde{Z}(i-1,t-1)}{\tilde{Z}(i,t-1)-\tilde{Z}(i-1,t-1)},\bar{\mu},\bar{\nu}-\bar{\mu})$-distributed (see \Cref{sub:2F1_distributions}); and when 
$\tilde{Z}(i,t-1)<\tilde{Z}(i-1,t)$,
$$
\tilde{Z}(i,t) = \tilde{Y} \tilde{Z}(i-1,t) + (1-\tilde{Y}) \tilde{Z}(i,t-1),
$$
where $\tilde{Y}$ is $\mathscr{NBBeta}^{-1}(2\bar{\nu}-1,\tfrac{\tilde{Z}(i,t-1)-\tilde{Z}(i-1,t-1)}{\tilde{Z}(i-1,t)-\tilde{Z}(i-1,t-1)},\bar{\mu},\bar{\nu})$-distributed.
\end{enumerate}

In the special case when $\bar{\nu}=1/2$, the recursion simplifies as follows: When $\tilde{Z}(i,t-1)>\tilde{Z}(i-1,t)$,
$$
\tilde{Z}(i,t) = \tilde{Y} \tilde{Z}(i,t-1) + (1-\tilde{Y}) \tilde{Z}(i-1,t),
$$
where $\tilde{Y}$ is $\mathscr{Beta}^{-1}(\bar{\mu},\tfrac{1}{2}-\bar{\mu})$-distributed; and when
$\tilde{Z}(i,t-1)<\tilde{Z}(i-1,t)$,
$$
\tilde{Z}(i,t) = \tilde{Y} \tilde{Z}(i-1,t) + (1-\tilde{Y}) \tilde{Z}(i,t-1),
$$
where $\tilde{Y}$ is $\mathscr{Beta}^{-1}(\bar{\mu},\tfrac{1}{2})$-distributed. 
\end{lemma}
\begin{proof}
	We only prove the general $\bar{\nu}$ recursion of $\tilde{Z}$ when $\tilde{Z}(i,t-1)>\tilde{Z}(i-1,t)$. The other case and specialization to $\bar{\nu}=\tfrac{1}{2}$ then follows likewise. 
Let $X$ be distributed as
$$
\mathscr{NBB}_1 \left(2\bar{\nu}-1, \frac{Z(i-1, t)^{-1} - Z(i-1,t-1)^{-1}}{Z(i,t-1)^{-1} - Z(i-1,t-1)^{-1}}, \frac{Z(i,t-1)}{Z(i-1, t)}, \bar{\mu}, \bar{\nu}-\bar{\mu}\right).
$$
From \eqref{NBB_recurrence}, it follows that
\begin{equation}\label{eq:Zceqn}
\frac{Z(i,t)}{Z(i,t-1)} = X\qquad \textrm{and}\qquad \frac{Z(i,t)}{Z(i-1,t)} = c \cdot X, \quad \textrm{where } c := \frac{Z(i,t-1)}{Z(i-1,t)}.
\end{equation}
Define $Y := \frac{X-c\cdot X}{1- c\cdot X}$. Since \eqref{NBB_recurrence} shows that $X$ is $\mathscr{NBB}_1$-distributed (with suitable parameters), we may employ \Cref{lem:changeofvar} to show that $Y$ is $\mathscr{NBBeta}(r,p,m,n)$ with $r=2\bar{\nu}-1$, $p= \frac{Z(i-1, t)^{-1} - Z(i-1,t-1)^{-1}}{Z(i,t-1)^{-1} - Z(i-1,t-1)^{-1}}$, $m= \bar{\mu}$ and $n = \bar{\nu}-\bar{\mu}$. By \eqref{eq:Zceqn}, 
$$
\frac{\frac{Z(i,t)}{Z(i,t-1)} - \frac{Z(i,t)}{Z(i-1,t)}}{1-\frac{Z(i,t)}{Z(i-1,t)}} = Y.
$$
We may rewrite things now via $\tilde{Z}$. In these variables, $p = \tfrac{\tilde{Z}(i-1,t)-\tilde{Z}(i-1,t-1)}{\tilde{Z}(i,t-1)-\tilde{Z}(i-1,t-1)}$ and the above recursion reduces to the desired relation
$$
\tilde{Z}(i,t) = \tilde{Y} \tilde{Z}(i,t-1) + (1-\tilde{Y}) \tilde{Z}(i-1,t),
$$
where $\tilde{Y} = Y^{-1}$. 
\end{proof}

\subsection{Convergence}
\label{sub:convergence_to_beta_model}

Let us now prove the convergence of the $q$-Hahn PushTASEP to the
process $Z(i,t)$ from \Cref{def:Z_process}.
\begin{theorem}
	\label{thm:beta_convergence}
	For fixed $\bar{\mu}$ and $\bar{\nu}$, as $\eps \to 0+$,
	the process $\{X(i,t)\colon 1\le i\le N,\; t\in \mathbb{Z}_{\ge0}\}$
	converges to $\{Z(i, t)\colon 1\le i\le N,\; t\in \mathbb{Z}_{\ge0}\}$.
\end{theorem}
The proof occupies the rest of the subsection.
We will use the following two facts
proven in
\cite{CorwinBarraquand2015Beta}
(Lemmas 2.2 and 2.3):
\begin{proposition}
\label{prop:q_to_1_convergence_BC16}
\mbox{}
	\begin{enumerate}[\bf1.\/]
		\item
			For $r, q \in (0, 1)$ and $x, y >0$,
			\begin{equation*}
				\frac{(r q^{y}; q)_{\infty}}{(r q^{x}; q)_{\infty}} \to (1-r)^{x - y} \qquad \text{as} \ q \to 1.
			\end{equation*}
		\item If $X_\eps$ is distributed as
			$\varphi_{e^{-\eps}, e^{-\bar{\mu} \eps}, e^{-\bar{\nu} \eps}}(\cdot \mid \infty)$,
			then $\exp(-\eps X_{\eps})$ converges in distribution as $\eps \to 0+$ to
			$\mathscr{B}_1(0, \bar{\mu}, \bar{\nu}-\bar{\mu})$.
	\end{enumerate}
\end{proposition}

Clearly, $X(i,0)=Z(i,0)=1$.
The second part of \Cref{prop:q_to_1_convergence_BC16} implies that
$X(1,t)=q^{-x_1(t)-1}$ converges to $Z(1,t)$, since the
first $q$-Hahn PushTASEP particle $x_1(t)$
follows a random walk with jump distribution
$\varphi_{q,\mu,\nu}(\cdot\mid \infty)$.

To complete the proof, we need to show that conditionally on
\begin{equation*}
	x_{i-1}(t-1) + i-1 \sim  \frac{\log X}{\eps},
	\qquad
	x_{i}(t-1)+i \sim  \frac{\log Y}{\eps},
	\qquad
	x_{i-1}(t)+i-1 \sim \frac{\log Z}{\eps},
\end{equation*}
\begin{enumerate}[\bf({case} 1)\/]
	\item If $Y<Z$,
		$X(i,t)/Y$ converges in distribution to
		$\mathscr{NBB}_1\left( 2\bar{\nu}-1, \frac{Z^{-1} - X^{-1}}{Y^{-1} - X^{-1}}, \frac{Y}{Z}, \bar{\mu}, \bar{\nu}-\bar{\mu}\right)$;
	\item If $Y>Z$,
		$X(i,t)/Z$
		converges in distribution to
		$\mathscr{NBB}_1\left( 2\bar{\nu}-1, \frac{Y^{-1} - X^{-1}}{Z^{-1} - X^{-1}}, \frac{Z}{Y}, \bar{\mu}, \bar{\nu} \right)$.
\end{enumerate}
As before, let us use the notation
\begin{equation*}
	\ell = x_{i-1}(t-1) - x_{i-1}(t), \qquad
	g = x_{i-1}(t-1) - x_{i}(t-1)-1,\qquad
	\quad\text{and}\quad
	L =   x_{i}(t-1) - x_{i}(t).
\end{equation*}
We will prove the above two cases separately
using formulas
\eqref{P_ell_g_rewrite_first_case_87},
\eqref{P_ell_g_rewrite_second_case_87}
for the update probabilities $\mathbf{P}_{\ell,g}(L)$.

\medskip\noindent\textbf{Proof of case 1.}
The case $Y<Z$ corresponds to representation
\eqref{P_ell_g_rewrite_first_case_87}
for $\mathbf{P}_{\ell,g}(L)$.
It suffices to show that for a fixed $\mathsf{t} > 0$,
\begin{multline}
	\label{limit_l<g}
	\lim_{\eps \to 0+}\eps^{-1}
	\mathbf{P}_{\ell,g}(\lceil \mathsf{t}/\eps \rceil)
	=
	\frac{e^{-\mathsf{t} \bar{\mu}}(1-e^{-\mathsf{t}})^{\bar{\nu}-\bar{\mu}-1}}
	{(1-e^{-\mathsf{t}}Y/Z)^{\bar{\nu}}}
	\frac{(1-Y/Z)^{\bar{\mu}}\Gamma(\bar{\nu})}{\Gamma(\bar{\mu}) \Gamma(\bar{\nu}-\bar{\mu})}
	\left(1-\frac{Z^{-1} - X^{-1}}{Y^{-1} - X^{-1}}\right)^{2\bar{\nu}-1}
	\\
	\times
	{_{2}F_{1}}\left(
		2\bar{\nu}-1, \bar{\nu}; \bar{\nu}-\bar{\mu}; \frac{X/Z - 1}{X/Y-1} \cdot \frac{1 - e^{-\mathsf{t}}}{1-e^{-\mathsf{t}}Y/Z}
	\right).
\end{multline}
Rewrite the product of the $q$-Pochhammer symbols preceding ${_{8}\phi_{7}}$
in the expression \eqref{P_ell_g_rewrite_first_case_87} as (in this case we use the notation
$L=\lceil \mathsf{t}/\eps \rceil$)
\begin{equation*}
	\frac{\mu^{L} (\nu/\mu; q)_{L} (\mu; q)_{\infty}}
	{(q;q)_{L}(\nu; q)_{\infty}}
	\,
	\frac{(\nu^{2}q^{g}; q)_{\infty}}
	{(q^{g+1}; q)_{\infty}}
	\,
	\frac{(q^{g-\ell+1}; q)_{\infty}}
	{(\mu \nu q^{g-\ell}; q)_{\infty}}
	\,
	\frac{(\nu q^{g-\ell}; q)_{\infty}}
	{(\nu^{2} q^{g-\ell}; q)_{\infty}}
	\,
	\frac{(\nu^{2} q^{L+g-\ell}; q)_{\infty}}
	{(\nu q^{L+g-\ell}; q)_{\infty}}.
\end{equation*}
For $L = \lceil \mathsf{t}/\eps \rceil$ the second part of \Cref{prop:q_to_1_convergence_BC16} implies
\begin{align*}
	\frac{\mu^{L} (\nu/\mu; q)_{L} (\mu; q)_{\infty}}{\eps(q;q)_{L}(\nu; q)_{\infty}}
	\to
	\frac{e^{-\mathsf{t} \bar{\mu}}(1-e^{-\mathsf{t}})^{\bar{\nu}-\bar{\mu}-1}\Gamma(\bar{\nu})}
	{\Gamma(\bar{\mu}) \Gamma(\bar{\nu}-\bar{\mu})},
\end{align*}
while the first part of \Cref{prop:q_to_1_convergence_BC16} leads to
\begin{align*}
	\frac{(\nu^{2}q^{g}; q)_{\infty}}
	{(q^{g+1}; q)_{\infty}}
	\to
	(1 - Y/X)^{1 - 2\bar{\nu}},
	\qquad &
	\frac{(q^{g-\ell+1}; q)_{\infty}}{(\mu \nu q^{g-\ell}; q)_{\infty}}
	\to
	(1-Y/Z)^{\bar{\mu}+\bar{\nu}-1}, \\
	\frac{(\nu q^{g-\ell}; q)_{\infty}}{(\nu^{2} q^{g-\ell}; q)_{\infty}} \to (1-Y/Z)^{\bar{\nu}},
	\qquad &
	\frac{(\nu^{2} q^{L+g-\ell}; q)_{\infty}}{(\nu q^{L+g-\ell}; q)_{\infty}} \to (1-e^{-\mathsf{t}}Y/Z)^{-\bar{\nu}}.
\end{align*}
The $k$-th term in the summation for ${_{8}\phi_{7}}$ in the expression \eqref{P_ell_g_rewrite_first_case_87} is
\begin{multline*}
	\frac{(q^{-\ell}; q)_{k}(q^{-L}; q)_{k}}
	{(\nu^{2}q^{L+g-\ell}; q)_{k} (\nu^{2}q^{g}; q)_{k}}
	\,
	\frac{(\nu^{2}q^{-1}; q)_{k}(\nu; q)_{k}}
	{(q; q)_{k}(\nu \mu^{-1}; q)_{k}}
	\left( \frac{q^{L+g+1}}{\mu}\right)^{k}
	\\\times
	\frac{(\nu^{2} q^{g-\ell-1}; q)_{k}(\nu q^{(g-\ell+1)/2}; q)_{k} (-\nu q^{(g-\ell+1)/2}; q)_{k}  (\mu \nu q^{g-\ell}; q)_{k} }
	{(\nu q^{(g-\ell-1)/2}; q)_{k} (- \nu q^{(g-\ell-1)/2}; q)_{k} (q^{g-\ell+1}; q)_{k}  (\nu q^{g-\ell}; q)_{k}}.
\end{multline*}
For fixed $k$ we have the following convergence:
\begin{align*}
	&
	\frac{(\nu^{2} q^{g-\ell-1}; q)_{k}(\nu q^{(g-\ell+1)/2}; q)_{k} (-\nu q^{(g-\ell+1)/2}; q)_{k}  (\mu \nu q^{g-\ell}; q)_{k} }
	{(\nu q^{(g-\ell-1)/2}; q)_{k} (- \nu q^{(g-\ell-1)/2}; q)_{k} (q^{g-\ell+1}; q)_{k}  (\nu q^{g-\ell}; q)_{k}}
	\to 1,
	\\
	&
	\frac{(q^{-\ell}; q)_{k}(q^{-L}; q)_{k}}
	{(\nu^{2}q^{L+g-\ell}; q)_{k} (\nu^{2}q^{g}; q)_{k}}
	\,
	\left(
		\frac{q^{L+g+1}}{\mu}
	\right)^{k}
	\to
	\frac{(X/Z-1)^{k} (e^{\mathsf{t}}-1)^{k}}{(1-e^{-\mathsf{t}}Y/Z)^{k}(1-Y/X)^{k}}\, (e^{-\mathsf{t}}Y/X)^{k},
	\\
	&\hspace{80pt}
	\frac{(\nu^{2}q^{-1}; q)_{k}(\nu; q)_{k}}{(q; q)_{k}(\nu \mu^{-1}; q)_{k}}
	\to
	\frac{(2 \bar{\nu}-1)_{k}(\bar{\nu})_{k}}{k!(\bar{\nu}-\bar{\mu})_{k}}.
\end{align*}
Hence the whole ${_{8}\phi_{7}}$
converges to
${_{2}F_{1}}\left(2\bar{\nu}-1, \bar{\nu}; \bar{\nu}-\bar{\mu}; \frac{X/Z - 1}{X/Y-1} \cdot \frac{1 - e^{-\mathsf{t}}}{1-e^{-\mathsf{t}}Y/Z}\right)$.
Combining everything together gives us \eqref{limit_l<g}, which establishes the first case.

\medskip\noindent\textbf{Proof of case 2.}
For the case $Y>Z$ we use representation
\eqref{P_ell_g_rewrite_second_case_87}
for $\mathbf{P}_{\ell,g}(L)$.
It suffices to show that for a fixed $\mathsf{t} > \log(Y/Z),$\footnote{This condition corresponds to $L\ge \ell-g$.}
\begin{multline}
\label{limit_l>g}
\lim_{\eps \to 0}
\eps^{-1}
\mathbf{P}_{\ell,g}(\lceil \mathsf{t}/\eps \rceil+\ell-g)
=
\frac{e^{-\mathsf{t} \bar{\mu}}(1-e^{-\mathsf{t}})^{\bar{\nu}-1}}{(1-e^{-\mathsf{t}}Z/Y)^{\bar{\nu}+\bar{\mu}}} \, \frac{(1-Z/Y)^{\bar{\mu}}\Gamma(\bar{\nu} + \bar{\mu})}{\Gamma(\bar{\mu}) \Gamma(\bar{\nu})}
\left(1-\frac{Y^{-1} - X^{-1}}{Z^{-1} - X^{-1}}\right)^{2\bar{\nu}-1}
\\\times
{_{2}F_{1}}\left(2\bar{\nu}-1, \bar{\nu}+\bar{\mu}; \bar{\nu}; \frac{X/Y - 1}{X/Z-1} \cdot \frac{1 - e^{-\mathsf{t}}}{1-e^{-\mathsf{t}}Z/Y}\right).
\end{multline}
Rewrite the product of $q$-Pochhammer symbols preceding ${_{8}\phi_{7}}$ in
\eqref{P_ell_g_rewrite_second_case_87} as
\begin{multline*}
	\frac{\mu^{L-\ell+g} (\nu/\mu; q)_{L-\ell+g} (\mu; q)_{\infty}}{(q;q)_{L-\ell+g}(\nu; q)_{\infty}}\, \frac{(\nu; q)_{\infty}^{2}}{(\mu \nu; q)_{\infty}(\nu/\mu; q)_{\infty}} \, \frac{(q^{\ell-g+1}; q)_{\infty}(q^{\ell-q}\nu/\mu; q)_{\infty}}{(\nu q^{\ell-g}; q)_{\infty}(\nu^{2}q^{\ell-g}; q)_{\infty}}
	\\\times
	\frac{(q^{L+g-\ell}\nu/\mu; q)_{\infty}}{(\nu q^{L+g-\ell}; q)_{\infty}} \, \frac{(\nu^{2}q^{L}; q)_{\infty}}{(q^{L}\nu/\mu; q)_{\infty}}\, \frac{(\nu^{2}q^{\ell}; q)_{\infty}}{(q^{\ell+1}; q)_{\infty}}.
\end{multline*}
With the notation $L-\ell+g = \lceil \mathsf{t}/\eps \rceil$, the second part of \Cref{prop:q_to_1_convergence_BC16} implies
\begin{align*}
	\frac{\mu^{L-\ell+g} (\nu/\mu; q)_{L-\ell+g} (\mu; q)_{\infty}}{\eps(q;q)_{L-\ell+g}(\nu; q)_{\infty}}
	\to
	\frac{e^{-\mathsf{t} \bar{\mu}}(1-e^{-\mathsf{t}})^{\bar{\nu}-\bar{\mu}-1}\Gamma(\bar{\nu})}{\Gamma(\bar{\mu}) \Gamma(\bar{\nu}-\bar{\mu})}.
\end{align*}
The first part of \Cref{prop:q_to_1_convergence_BC16} implies
that
\begin{align*}
	\frac{(q^{\ell-g+1}; q)_{\infty}}{(\nu q^{\ell-g}; q)_{\infty}} \to (1-Z/Y)^{\bar{\nu}-1},
	\qquad &
	\frac{(q^{\ell-q}\nu/\mu; q)_{\infty}}{(\nu^{2}q^{\ell-g}; q)_{\infty}} \to (1-Z/Y)^{\bar{\mu}+\bar{\nu}},
	\\
	\frac{(\nu^{2}q^{\ell}; q)_{\infty}}{(q^{\ell+1}; q)_{\infty}} \to (1-Z/X)^{1 - 2\bar{\nu}},
	\qquad&
	\frac{(q^{L+g-\ell}\nu/\mu; q)_{\infty}}{(\nu q^{L+g-\ell}; q)_{\infty}} \to (1-e^{-\mathsf{t}})^{\bar{\mu}},
	\\\
	&
	\frac{(\nu^{2}q^{L}; q)_{\infty}}{(q^{L}\nu/\mu; q)_{\infty}} \to (1-e^{-\mathsf{t}}Z/Y)^{-\bar{\mu}-\bar{\nu}}.
\end{align*}

Recall the $q$-Gamma function
\begin{align*}
	\Gamma_{q}(x)=\frac{(q; q)_{\infty}}{(q^{x}; q)_{\infty}}(1-q)^{1-x},
\end{align*}
which converges to the ordinary Gamma function as $q\nearrow 1$.
Then
\begin{align*}
	\frac{(\nu; q)_{\infty}^{2}}{(\mu \nu; q)_{\infty}(\nu/\mu; q)_{\infty}}
	=
	\frac{\Gamma_{q}(\bar{\mu} + \bar{\nu})\Gamma_{q}(\bar{\nu}-\bar{\mu})}
	{\Gamma_{q}(\bar{\nu})^{2}}
	\to
	\frac{\Gamma(\bar{\mu} + \bar{\nu})\Gamma(\bar{\nu}-\bar{\mu})}{\Gamma(\bar{\nu})^{2}}.
\end{align*}
The $k$-th term in the summation for ${_{8}\phi_{7}}$ in \eqref{P_ell_g_rewrite_second_case_87} has the form:
\begin{multline*}
	\frac{(q^{-g}; q)_{k}(q^{-L+\ell-g}; q)_{k}}{(\nu^{2}q^{L}; q)_{k} (\nu^{2}q^{\ell}; q)_{k}}
	\,
	\frac{(\nu^{2}q^{-1}; q)_{k}(\nu \mu; q)_{k}}{(q; q)_{k}(\nu; q)_{k}}
	\,
	 \left( \frac{q^{L+g+1}}{\mu}\right)^{k}
	\\
	\times
	\frac{
		(\nu^{2} q^{\ell-g-1}; q)_{k}(\nu q^{(\ell-g+1)/2}; q)_{k}
		(-\nu q^{(\ell-g+1)/2}; q)_{k}(\nu q^{\ell-g}; q)_{k}
	}
	{
		(\nu q^{(\ell-g-1)/2}; q)_{k}
		(-\nu q^{(\ell-g-1)/2}; q)_{k}
		(q^{\ell-g+1}; q)_{k}(\nu q^{\ell-g}/\mu; q)_{k}
	}.
\end{multline*}
For fixed $k$ we have the following behavior:
\begin{align*}
	&\frac{
		(\nu^{2} q^{\ell-g-1}; q)_{k}(\nu q^{(\ell-g+1)/2}; q)_{k}
		(-\nu q^{(\ell-g+1)/2}; q)_{k}(\nu q^{\ell-g}; q)_{k}
	}
	{
		(\nu q^{(\ell-g-1)/2}; q)_{k} (-\nu q^{(\ell-g-1)/2}; q)_{k} (q^{\ell-g+1}; q)_{k}(\nu q^{\ell-g}/\mu; q)_{k}
	}
	\to 1,
	\\
	&
	\frac{(q^{-g}; q)_{k}(q^{-L+\ell-g}; q)_{k}}{(\nu^{2}q^{L}; q)_{k} (\nu^{2}q^{\ell}; q)_{k}}
	\,
	\left( \frac{q^{L+g+1}}{\mu}\right)^{k}
	\to
	 \frac{(X/Y-1)^{k} (e^{\mathsf{t}}-1)^{k}}{(1-e^{-\mathsf{t}}Z/Y)^{k}(1-Z/X)^{k}}
	 \,
	(e^{-\mathsf{t}}Z/X)^{k},
	\\ &
	\hspace{70pt}
	\frac{(\nu^{2}q^{-1}; q)_{k}(\nu \mu; q)_{k}}{(q; q)_{k}(\nu; q)_{k}}
	\to
	\frac{(2 \bar{\nu}-1)_{k}(\bar{\nu}+\bar{\mu})_{k}}
	{k!(\bar{\nu})_{k}}.
\end{align*}
Hence the whole ${_{8}\phi_{7}}$
expression in \eqref{P_ell_g_rewrite_second_case_87}
converges to ${_{2}F_{1}}\left(2\bar{\nu}-1, \bar{\nu}+\bar{\mu}; \bar{\nu}; \frac{X/Y-1}{X/Z-1} \cdot
\frac{1 - e^{-\mathsf{t}}}{1-e^{-\mathsf{t}}Z/Y}\right)$.
Combining everything together gives us \eqref{limit_l>g}.
This completes the proof of \Cref{thm:beta_convergence}.

\subsection{Contour integral observables of the beta model}
\label{sub:beta_contour_integrals}

The nested contour integral expressions for the $q$-moments of the
$q$-Hahn TASEP produce (in the $q\to1$ scaling limit) contour
integral observables for the process $\{Z(i,t)\}$.
For $\vec n=(n_1\ge n_2\ge \ldots\ge n_k\ge0 )$ and $t\in \mathbb{Z}_{\ge0}$
define $U(t;\vec n):=\mathop{\mathbb{E}}\bigl[\prod_{i=1}^{k} Z(n_i,t)^{-1}\bigr]=\mathop{\mathbb{E}}\bigl[\prod_{i=1}^{k} \tilde{Z}(n_i,t)\bigr]$.
\begin{proposition}
	\label{prop:beta_moments}
	When $\bar{\mu}>k$, we have
	\begin{equation}\label{beta_contour_moments}
		U(t;\vec n)=
		\frac{1}{(2\pi\sqrt{-1})^k}
		\int \ldots\int
		\prod_{1\le A<B\le k}\frac{w_A-w_B}{w_A-w_B-1}
		\prod_{j=1}^{k}
		\left(\frac{\bar{\nu}+w_j}{w_j} \right)^{n_j}
		\left( \frac{\bar{\nu}-1-w_j}{\bar{\mu}-1-w_j} \right)^{t}
		\frac{dw_j}{\bar{\nu}+w_j}.
	\end{equation}
	Here
	the contours are simple closed curves around $0$
	which do not encircle
	$\bar{\mu}-1$
	or $-\bar{\nu}$,
	and such that the $w_A$ contour encircles the $w_B+1$ one for all $A<B$.
\end{proposition}
\begin{proof}
	\Cref{thm:beta_convergence}
	implies that $u(t;\vec n)\to U(t;\vec n)$
	under the scaling \eqref{beta_parameters_scaling}.
	Let $w_i$ be the contours as in \eqref{beta_contour_moments},
	and set $z_j=q^{w_j}=e^{-\eps w_j}$.
	Then the contours $z_j$ are exactly the ones in
	\Cref{thm:push_q_moments_in_the_text}. As $\eps\to0$, we have the following convergence in the integrand:
	\begin{align*}
	\frac{z_A-z_B}{z_A- q z_{B}}
	\to
	\frac{w_{A}-w_{B}}{w_{A}-w_{B}-1}, &
	\qquad
	\frac{1-\nu z_{j}}{1-z_{j}}
	\to
	\frac{\bar{\nu} + w_{j}}{w_{j}},
	\\
	\frac{1- \nu q^{-1} z_{j}^{-1}}{1- \mu q^{-1} z_{j}^{-1}}
	\to
	\frac{\bar{\nu}-1-w_{j}}{\bar{\mu}-1-w_{j}}, &
	\qquad
	\frac{d z_{j}}{z_{j}(1-\nu z_{j})}
	\to
	 -\frac{d w_{j}}{\bar{\nu} + w_{j}}.
	\end{align*}
	This completes the proof.
	Note that the restriction $\bar{\mu}>k$ in \eqref{beta_contour_moments}
	comes from $\mu<q^k$ in \Cref{thm:push_q_moments_in_the_text}.
\end{proof}

Again, using the moments of \Cref{prop:beta_moments}
or taking the scaling limit as $q\nearrow1$ of \Cref{conj:conjecture_push_Fredholm},
we can write down a conjectural Fredholm determinantal
expression for the Laplace transform of $Z(n,t)$:
\begin{conjecture}
	\label{conj:conjecture_push_Fredholm_beta}
	For $\xi\in \mathbb{C}\setminus\mathbb{R}_{>0}$,
	we have
	\begin{equation*}
		\mathop{\mathbb{E}}
		\Bigl[
			e^{\xi Z(n,t)^{-1}}
		\Bigr]=
        \mathop{\mathbb{E}}
		\Bigl[
			e^{\xi \tilde{Z}(n,t)}
		\Bigr]=
		\det(I+K^{\mathscr{B}}_\xi),
	\end{equation*}
	where $K^{\mathscr{B}}_\xi$ is a kernel of an integral operator
	on a small circle around $0$:
	\begin{equation*}
		K^{\mathscr{B}}_\xi(v,v')=
		\frac{1}{2\pi\sqrt{-1}}
		\int_{-\infty\sqrt{-1}+\frac12}^{\infty\sqrt{-1}+\frac12}
		\frac{\pi}{\sin(\pi s)}(-\xi)^{s}\,
		\frac{g^\mathscr{B}(v)}{g^{\mathscr{B}}(v+s)}
		\,
		\frac{ds}{s+v-s'},
	\end{equation*}
	where
	\begin{equation*}
		g^{\mathscr{B}}(v)=
		\left( \frac{\Gamma(v)}{\Gamma(v+\bar\nu)} \right)^n
		\left( \frac{\Gamma(\bar{\nu}-w)}{\Gamma(\bar{\mu}-w)} \right)^t
		\Gamma(v+\bar{\nu})
		.
	\end{equation*}
\end{conjecture}

\appendix

\section{Probability distributions from $q$-hypergeometric series}
\label{sec:q_hyp}

\subsection{Basic definitions}
\label{sub:q_hyp_notation}

Here we recall some basic facts about $q$-hypergeometric series.
Define the $q$-Pochhammer symbols
\begin{equation*}
	(a; q)_{n} =
	\begin{cases}
		1,&n=0;\\
		\prod_{i=1}^{n}(1 - a q^{i-1}), &n\ge1;\\
		\prod_{i=n}^{-1}(1-aq^{i})^{-1},&n\le -1,
	\end{cases}
	\qquad \text{and} \qquad
	(a; q)_{\infty} = \prod_{i=1}^{\infty}(1 - a q^{i-1}).
\end{equation*}
For the definition of the infinite $q$-Pochhammer symbol we assume $|q| < 1$.

The unilateral basic hypergeometric series $_{k+1}\phi_{k}$ is defined via
\begin{equation}\label{q_hyp_defn}
_{k+1}\phi_{k} \left[\begin{array}{ccc} a_{1} & \ldots & a_{k+1} \\ b_{1} & \ldots & b_{k} \end{array}; q, z \right] = \sum_{n=0}^{\infty}
\frac{(a_{1}, \ldots, a_{k+1}; q)_{n}}{(b_{1}, \ldots, b_{k}, q; q)_{n}} z^{n},
\end{equation}
where $(c_{1}, \ldots, c_{m}; q)_{n} = (c_1;q)_n\ldots (c_m;q)_n $.
If one of $a_{j}$ is $q^{-y}$ for a positive integer $y$, then this series is
terminating. Otherwise we assume  $|q|, |z| < 1$ for the sum to be convergent.

In \Cref{sub:phi_distr,sub:q_hyp} below we describe
two families of probability distributions
with weights given in terms of $q$-Pochhammer symbols.
Their normalization constants are computed by
applying $q$-summation identities.

\subsection{$q$-beta-binomial distribution}
\label{sub:phi_distr}

For integers $0 \leq s \leq y$ define
\begin{equation}\label{phi_qmunu_definition}
	\varphi_{q,\mu,\nu}(s\mid y):=\mu^{s}\frac{(\nu/\mu;q)_{s}(\mu;q)_{y-s}}{(\nu;q)_{y}}\frac{(q;q)_{y}}{(q;q)_{s}(q;q)_{y-s}}.
\end{equation}
\begin{lemma}
	\label{lemma:phi_sum_to_one}
	For any nonnegative integer $y$ we have
	\begin{equation} \label{phi_qmunu_sum}
		\sum_{s=0}^{y} \varphi_{q,\mu,\nu}(s\mid y)=1.
	\end{equation}
\end{lemma}
Identity \eqref{phi_qmunu_sum} first appeared in the context of interacting particle systems in \cite{Povolotsky2013}.
\begin{proof}[Proof of \Cref{lemma:phi_sum_to_one}]
	Use Heine's $q$-generalization of Gauss' summation formula \cite[(II.8)]{GasperRahman},
	\begin{equation} \label{q-Gauss}
		_{2}\phi_{1}
		\left[
			\begin{array}{cc} a & b \\ \multicolumn{2}{c}{c} \end{array};
			q, c/ab
		\right] = \frac{(c/a; q)_{\infty}(c/b; q)_{\infty}}{(c; q)_{\infty}(c/ab; q)_{\infty}}.
	\end{equation}
	Take $a = q^{-y}$, $b= \mu/\nu$, $c=\mu q^{1-y}$
	in \eqref{q-Gauss}. This makes the $_2\phi_1$ function terminating, and the
	resulting finite summation identity is simply
	\eqref{phi_qmunu_sum} with $q$ replaced by $q^{-1}$.
\end{proof}

Therefore, for all values of the parameters $(q,\mu,\nu)$ for which
$\varphi_{q,\mu,\nu}(s\mid y)$ is well-defined and nonnegative for every $0 \leq s
\leq y$, \eqref{phi_qmunu_definition} is a probability distribution on
$\{0, 1, \ldots, y\}$.
One such family of parameters is $0\le q<1$, $0 \le\mu <1$, $\nu \le \mu$.
Another choice  leading to a probability distribution is $q > 1$, $\mu =
q^{-m}$, $\nu = q^{-n}$ for nonnegative integers $m$, $n$ with $m \le n$, $y
\le n$.

We can also take $y \to \infty$ to get the function
\begin{equation}
	\label{phi_qmunu_infinity_definition}
	\varphi_{q,\mu,\nu}(s\mid \infty):=\mu^{s}\frac{(\nu/\mu;q)_{s}}{(q;q)_{s}}\frac{(\mu;q)_{\infty}}{(\nu; q)_{\infty}},
\end{equation}
which for appropriate values of parameters is a probability distribution on $\mathbb{Z}_{\ge0}$.

The distribution $\varphi_{q,\mu,\nu}$
appears (under a simple change of parameters,
see
\cite[Section 5.2]{BCPS2014}
for details)
as the orthogonality weight of the classical
$q$-Hahn orthogonal polynomials \cite[Section 3.6]{Koekoek1996}.
It is also related to a very natural $q$-deformation
of the Polya urn scheme \cite{Gnedin2009}.
As such, we call $\varphi_{q,\mu,\nu}$ the
\emph{$q$-beta-binomial distribution}.

By taking $q=e^{-\eps},
\mu=e^{-\alpha\eps}, \nu=e^{-(\alpha+\beta)\eps}$ and letting
$\eps\to 0+$, we see that $\varphi_{q,\mu,\nu}(s\mid y)$ converges to
\begin{align*}
	\frac{\Gamma(\alpha+y-s)\Gamma(\beta+s)\Gamma(\alpha+\beta)\Gamma(y+1)}{\Gamma{(\alpha)}\Gamma(\beta)\Gamma(\alpha+\beta+y)\Gamma(s+1)\Gamma(y-s+1)},
\end{align*}
which is the probability of $s$ under the beta-binomial distribution with parameters $y, \alpha, \beta$.
The beta-binomial distribution is
the orthogonality weight for the Hahn
orthogonal polynomials \cite[Section 1.5]{Koekoek1996}, and also arises from the ordinary
Polya urn scheme.

Another property of the $q$-beta-binomial distribution which we need is
the following symmetry:
\begin{lemma}
	\label{lemma:phi_symmetry}
	For any nonnegative integers $x$ and $y$ we have
	\begin{equation}
		\label{phi_qmunu_interchange}
		\sum_{s=0}^{y}\varphi_{q,\mu,\nu}(s\mid y)\,q^{sx}
		=
		\sum_{t=0}^{x}\varphi_{q,\mu,\nu}(t\mid x)\,q^{ty}.
	\end{equation}
\end{lemma}
\begin{proof}
	This is \cite[Proposition 1.2]{Corwin2014qmunu}, see also \cite{Barraquand_qhahn_2014}.
\end{proof}

For $x=0$ identity \eqref{phi_qmunu_interchange}
reduces to \eqref{phi_qmunu_sum}.

\subsection{$q$-hypergeometric distribution}
\label{sub:q_hyp}

For generic values of $a,b,c$
such that $a,b<1$, $c,q\in(0,1)$ and $\frac{c}{ab}\in(0,1)$,
the individual terms in the summation identity
\eqref{q-Gauss} are all nonnegative. Therefore, this identity
defines a probability distribution
\begin{equation}
	\label{q_hyp_distribution_definition}
	\psi_{q,a,b,c}(p):=
	\left(\frac{c}{ab} \right)^{p} \frac{(a; q)_{p}(b; q)_{p}}{(q;q)_{p}(c; q)_{p}}
	\frac{(c; q)_{\infty}(c/ab; q)_{\infty}}{(c/a; q)_{\infty}(c/b; q)_{\infty}}
\end{equation}
on the set of all nonnegative integers $p$.
We call it the \emph{$q$-hypergeometric distribution}
by analogy with the classical hypergeometric distribution whose
probability generating function $\sum_{p=0}^{\infty}z^p \mathop{\mathrm{Prob}}(X=p)$
is the Gauss hypergeometric function $_2F_1$.

\subsection{Distributions for the $q\to1$ beta limit}
\label{sub:2F1_distributions}

In \Cref{sec:beta_limit} we use several distributions
which we define here. Let the negative binomial distribution be
\begin{equation}
	\label{neg_binom_def}
	\mathscr{NB}(r,p)[k]
	=(1-p)^r\,\frac{p^k(r)_k}{k!},\qquad k=0,1,2,\ldots ;
	\qquad
	r\ge0,\quad 0\le p<1,
\end{equation}
where $(r)_k=r(r+1)\ldots(r+k-1)$ is the ordinary Pochhammer symbol.
Here
$\mathscr{NB}(r,p)[k]$ (and similar expressions below) stands for the probability weight of $k$
(or the probability density function in the absolutely continuous case),
and
$\mathscr{NB}(r,p)$ is the corresponding
random variable.
The generalized beta distribution of the first kind has the density
\begin{equation}
	\label{gen_beta_def}
	\mathscr{B}_1(c,m,n)[x]=
	\frac{(1-c)^m \Gamma(m+n)}{\Gamma(m)\Gamma(n)}\,
	\frac{x^{m-1}(1-x)^{n-1}}{(1-c x)^{m+n}},
	\qquad
	0<x<1,
\end{equation}
where $m,n>0$ and $c<1$. A special case of this distribution is the standard beta, denoted by $\mathscr{Beta}(m,n)$, which occurs when $c=0$. If $X$ is $\mathscr{Beta}(m,n)$-distributed, then we say that $X^{-1}$ is $\mathscr{Beta}^{-1}(m,n)$-distributed. (Note that this does not mean that the density of $X^{-1}$ is the inverse of the density of $X$.)

Combine the distributions \eqref{neg_binom_def} and \eqref{gen_beta_def}
and define the continuous distribution $\mathscr{NBB}_1(r,p,c,m,n)$ on $(0,1)$
as $\mathscr{B}_1(c,m,n+k)$, with $k\sim \mathscr{NB}(r,p)$.
That is, $\mathscr{NBB}_1$ has the density
\begin{equation}
	\label{nbb1_distribution_definition}
	\mathscr{NBB}_1(r,p,c,m,n)[x]
	=
	( 1 - p ) ^ { r }\,
	\frac { ( 1 - c ) ^ { m } \Gamma ( m + n ) } { \Gamma ( m ) \Gamma ( n ) }
	\,
	\frac { x ^ { m - 1 } ( 1 - x ) ^ { n - 1 } } { ( 1 - c x ) ^ { m + n } }
	\,
	{}_2F_1 \left( r , m + n ; n ; \frac { p ( 1 - x ) } { 1 - c x } \right),
\end{equation}
where ${}_2F_1$ is
the ordinary Gauss hypergeometric function
\begin{equation*}
	_2F_1(a,b;c;z)=\sum_{k=0}^{\infty}\frac{(a)_k(b)_k}{(c)_k}\frac{z^k}{k!}.
\end{equation*}
When $c=0$, this distribution reduces to the negative binomial beta
distribution which we denote by $\mathscr{NBBeta}(r,p,m,n)$. If $X$ is
$\mathscr{NBBeta}(r,p,m,n)$-distributed, then we say that $X^{-1}$ is
$\mathscr{NBBeta}^{-1}(r,p,m,n)$-distributed.

The next lemma shows that via a $c$-dependent linear fractional transform, these random variables can be made independent of $c$.
\begin{lemma}\label{lem:changeofvar}
If $X$ is distributed as $\mathscr{B}_1(c,m,n)$ (i.e., a generalized beta random variable with density on $[0,1]$ given by \eqref{gen_beta_def}), then $Y =\frac{X-c X}{1-c X}$ is $\mathscr{Beta}(m,n)$-distributed. Likewise if $W$ is distributed as $\mathscr{NBB}_1(r,p,c,m,n)$ (i.e., a random variable with density on $[0,1]$ given by \eqref{nbb1_distribution_definition}), then $V =\frac{W-c W}{1-c W}$ is $\mathscr{NBBeta}(r,p,m,n)$-distributed.
\end{lemma}
\begin{proof}
This follows from a simple change of variables applied to the densities.
\end{proof}

\printbibliography

\bigskip

\textsc{I. Corwin, Columbia University, Department of Mathematics, 2990 Broadway, New York, NY 10027, USA}

E-mail: \texttt{ivan.corwin@gmail.com}

\bigskip

\textsc{K. Matveev, Department of Mathematics, Brandeis University, 415 South Street, Waltham, MA, USA}

E-mail: \texttt{kosmatveev@gmail.com}

\bigskip

\textsc{L. Petrov, University of Virginia, Department of Mathematics, 141 Cabell Drive, Kerchof Hall, P.O. Box 400137, Charlottesville, VA 22904, USA, and Institute for Information Transmission Problems, Bolshoy Karetny per. 19, Moscow, 127994, Russia}

E-mail: \texttt{lenia.petrov@gmail.com}

\end{document}